\documentclass[3p,times,11pt,leqno]{elsarticle}

\usepackage{microtype}
\usepackage[english]{babel}

\usepackage{stmaryrd}
\usepackage{latexsym}
\usepackage{alltt}
\usepackage{amssymb}
\usepackage{amsmath}
\usepackage{fleqn}
\usepackage[all]{xy}
\SelectTips{eu}{}
\usepackage{booktabs}
\usepackage{xcolor}

\usepackage{pdfsync}
\usepackage[active]{srcltx}

\usepackage[colorlinks]{hyperref}

\usepackage{calc}
\newlength\nestedcwidth
\newcommand{\nested}[3]{%
  \setlength{\nestedcwidth}%
  {\maxof{\widthof{#1}}{\maxof{\widthof{#2}}{\widthof{#3}}}}%
  \fbox{\parbox{\nestedcwidth+4\fboxsep+4\fboxrule}{%
      \centering#3\\[.25\baselineskip]
      \fbox{\parbox{\nestedcwidth+2\fboxsep+2\fboxrule}{%
          \centering#2\\[.25\baselineskip]
          \fbox{\parbox{\nestedcwidth}{\centering#1}}%
        }}%
    }}%
}

\usepackage{amsmath,amssymb}

\ifluatex

\usepackage{newtxtext}
\usepackage{tgtermes}
\usepackage[T1]{fontenc}
\usepackage{textcomp}
\usepackage{newtxmath}

\usepackage{textcase}

\fi

\usepackage{amsthm}

\usepackage{tikz}
\usetikzlibrary{positioning}
\usetikzlibrary{calc}
\usetikzlibrary{arrows}
\usetikzlibrary{intersections}
\usetikzlibrary{backgrounds}

\tikzstyle{reverseclip}=[insert path={(-99cm, -99cm) rectangle (99cm, 99cm)}]
\tikzstyle{framed}=[draw,fill=white,drop shadow,inner sep=1.3ex,outer sep=5pt]

\usepackage[inline]{enumitem}
\setlist[itemize]{
  topsep=.5em,
  parsep=2pt,
  itemsep=.5\parskip,
  labelsep=.5em
}

\setlist[enumerate,1]{
  topsep=.5em,
  parsep=2pt,
  itemsep=.5\parskip,
  labelsep=.5em
}

\setlist[description]{
  font=\normalfont\itshape
}

\raggedright 

\newtheorem{definition}{Definition}[section]
\newtheorem{example}{Example}[section]

\newtheorem{proposition}{Proposition}[section]
\newtheorem{fact}{Fact}[section]
\newtheorem{lemma}{Lemma}[section]

\newtheorem{corollary}{Corollary}[section]

\newcommand\pto{\mathrel{\ooalign{\hfil$\mapstochar$\hfil\cr$\to$\cr}}}

\newcommand{\co}{\,\colon\;}
\newcommand{\ra}{\rightarrow}
\newcommand{\Ra}{\Rightarrow}

\newcommand{\st}{\mathrm{st}}
\newcommand{\op}{\mathrm{op}}

\newcommand{\T}{\mathcal{T}\!\!}
\newcommand{\impl}{\Rightarrow}

\newcommand{\mpair}[2]{\langle {#1},\,{#2} \rangle} 

\newcommand{\vsp}{\vspace{.5em}}

\def\C{\boldsymbol{C}}
\def\P{\mathcal{P}}
\def\I{\mathcal{I}}
\newcommand{\Mod}{\mathit{Mod}}
\newcommand{\Sen}{\mathit{Sen}}
\newcommand{\Sign}{\mathit{Sign}}

\newcommand{\Th}{\mathit{Th}}

\def\Set{\boldsymbol{SET}}
\def\Setp{\boldsymbol{P\!f\!n}}
\def\CAT{\boldsymbol{C\!AT}}
\def\3/2{\frac{3}{2}}
\newcommand{\dom}[1]{\Box{#1}}
\newcommand{\cod}[1]{{#1}\Box}
\def\DOM{\mathrm{dom}}

\def\PL{\mathcal{P\!\!L}}
\def\MSA{\mathcal{M\!S\!A}}

\usepackage{array}
\usepackage{arydshln}

\usepackage{setspace}

\DeclareSymbolFont{ams}{U}{msa}{m}{n}
\DeclareSymbolFontAlphabet{\mathams}{ams}

\begin{document}


\title{Generic partiality for $\frac{3}{2}$-institutions}

\author{R\u{a}zvan Diaconescu}
\ead{Razvan.Diaconescu@imar.ro}
\address{Simion Stoilow Institute of Mathematics of the Romanian Academy}
\date{\today}

\begin{abstract}
\noindent
$\3/2$-institutions have been introduced as an extension of
institution theory that accommodates implicitly partiality of the
signature morphisms together with its syntactic and semantic effects. 
In this paper we show that ordinary institutions that are equipped
with an inclusion system for their categories of signatures generate
naturally $\3/2$-institutions with \emph{explicit} partiality for
their signature morphisms.
This provides a general uniform way to build $\3/2$-institutions for
the foundations of conceptual blending and software evolution. 
Moreover our general construction allows for an uniform derivation of
some useful technical properties.
\end{abstract}

\maketitle
\section{Introduction}

\subsection{Institution theory}

The broad mathematical context of our work is the theory of institutions
\cite{Goguen-Burstall:Institutions-1992} which is a three-decades-old
category-theoretic abstract model theory that traditionally has been
playing a crucial foundational role in formal
specification(e.g. \cite{sannella-tarlecki-book}).   
It has been introduced in
\cite{Goguen-Burstall:Introducing-institutions-1983} as an answer to
the explosion in the number of population of logical systems there, as 
a very general mathematical study of formal logical systems, with
emphasis on semantics (model theory), that is not committed to any
particular logical system.
Its role has gradually expanded to other areas of logic-based computer
science, most notably to declarative programming and ontologies. In
parallel, and often in interdependence to its role in computer
science, in the past fifteen years it has made important contributions
to model theory through the new area called
\emph{institution-independent model theory} \cite{iimt} -- an abstract
approach to model theory that is liberated from any commitment to
particular logical systems.  
Institutions thus allowed for a smooth, systematic, and uniform
development of model theories for unconventional logical systems, as
well as of logic-by-translation techniques and of heterogeneous
multi-logic frameworks.  

Mathematically, institution theory is based upon a category-theoretic
\cite{MacLane:Categories-for-the-Working-Mathematician-1998}
formalization of the concept of logical system that includes the
syntax, the semantics, and the satisfaction relation between them. 
As a form of abstract model theory, it is the only one that
treats all these components of a logical system fully abstractly. 
In a nutshell, the above-mentioned formalization is a
category-theoretic structure \((\Sign, \Sen, \Mod, \models)\), called
\emph{institution}, that consists of 
\begin{enumerate*}[label=(\textit{\alph*})]
  
\item a category \(\Sign\) of so-called \emph{signatures},
  
\item two functors, \(\Sen \colon \Sign \to \Set\) for the syntax,
  given by sets of so-called \emph{sentences}, and \(\Mod \colon
  \Sign^{\varominus} \to \CAT\) for the semantics, given by categories of
  so-called \emph{models}, and 
  
\item for each signature \(\Sigma\), a binary \emph{satisfaction
    relation} \(\models_{\Sigma}\) between the \(\Sigma\)-models,
  i.e.\ objects of \(\Mod(\Sigma)\), and the \(\Sigma\)-sentences,
  i.e.\ elements of \(\Sen(\Sigma)\), 

\end{enumerate*}
such that for each morphism \(\varphi \colon \Sigma \to \Sigma'\) in
the category \(\Sign\), each \(\Sigma'\)-model \(M'\), and each
\(\Sigma\)-sentence \(\rho\) the following \emph{Satisfaction
  Condition} holds: 
\[
  M' \models_{\Sigma'} \Sen(\varphi)(\rho)
  \qquad \text{if and only if} \qquad
  \Mod(\varphi)(M') \models_{\Sigma} \rho.
\]

\vspace{-\parskip}
Because of its very high level of abstraction, this definition
accommodates not only well established logical systems but also very
unconventional ones. Moreover, it has served and it may serve as a
template for defining new ones. 
Institution theory approaches logic and model theory from a
relativistic, non-substantialist perspective, quite different from the 
common reading of formal logic. 
This does not mean that institution theory is opposed to the
established logic tradition, since it rather includes it from a higher
abstraction level. 
In fact, the real difference may occur at the level of the development
methodology: top-down in the case of institution theory, versus
bottom-up in the case of traditional logic. 
Consequently, in institution theory, concepts come naturally as
presumed features that a logical system might exhibit or not, and are
defined at the most appropriate level of abstraction; in developing
results, hypotheses are kept as general as possible and introduced on
a by-need basis. 

\subsection{$\3/2$-institutions}

In spite of the broad conceptual coverage provided by institution
theory there are specific aspects that require a general treatment but
that cannot be addressed by the ordinary concept of institution. 
This situation has lead to a number of extensions of standard
institution theory, such as towards many-valued truth
(\emph{$\mathcal{L}$-institutions} \cite{GradedConseq}), implicit Kripke
semantics (\emph{stratified institutions} \cite{strat,KripkeStrat}), etc.  
The most recent such extension are the \emph{$\3/2$-institutions} of
\cite{3/2Inst} that accomodate \emph{implicitly} partiality at the
level of the signature morphisms. 
Signature morphisms that are partial are very difficult to digest from
the perspective of logic for several reasons.
Firstly, conventional mathematical logic does not usually involve
signature morphisms at all, only very rarely in the form extensions
of languages (the term ``language'' often used in conventional 
mathematical logic corresponds to ``signature'' in our terminology). 
It was precisely specification theory that showed the need to consider
signature morphisms that are not necessarily inclusions or even
injective. 
Secondly, even within the context of specification theory the idea
that translating or mapping between signatures can be a partial has
hardly been considered at all.
A very notable exception to this is Goguen's research on algebraic
semiotics \cite{Goguen:Algebraic-Semiotics-1999} and conceptual
blending \cite{Goguen:What-Is-a-Concept-2005}. 
This work, that constitutes a mathematical and computational response
to the seminal proposal of Fauconnier and Turner
\cite{Fauconnier-Turner:Conceptual-Integration-Networks-1998} of
conceptual blending as a fundamental cognitive operation of language
and common-sense, has received much attention within the context of
the recent COINVENT project
\cite{Schorlemmer-Smaill-Kuhnberger-Kutz-Colton-Cambouropoulos-Pease:COINVENT-2014}.  
A serious shortcoming of the Goguen-COINVENT approach to conceptual
blending is a lack of an explicit semantic component, and the concept
of  $\3/2$-institutions have been proposed precisely as a remedy to
this. 
Moreover in \cite{3/2Inst} it is argued that partiality of signature
morphisms occurs naturally in the mathematical studies of merging of
software changes; this can be considered as another application area
for $\3/2$-institutions.  

While $\3/2$-institutions propose an \emph{implicit} approach to
partiality of signature morphisms and of their effects on the syntax
and of the semantics, many of the examples of $\3/2$-institutions in
\cite{3/2Inst} display a common pattern in the way they are derived
from ordinary institutions on the basis of explicit partiality of
signature morphisms. 
Here we explain this pattern by developing a generic method to
construct $\3/2$-institutions from ordinary institutions that in
essence requires only that the category of the signatures of the
institution is equipped with an \emph{inclusion system}
\cite{modalg,iimt}. 
Furthermore we exploit this general construction developing general
but crucial results on the existence of lax cocones and on model
amalgamation properties in $\3/2$-institutions, obtained on the basis
of corresponding properties of the underlying ordinary institutions.  

\subsection{Contributions and Structure of the Paper}

The paper is structured as follows:
\begin{enumerate}

\item In a preliminary section we review the theory of
  $\3/2$-institution introduced in \cite{3/2Inst}. 
 
\item A crucial section is dedicated to the development of categories of
  `partial maps' based upon inclusion systems \cite{modalg,iimt}. 
   This topic is well understood in the category theoretic literature,
   however the novelty here is that we do this on the basis on
   inclusion systems rather than factorisation systems (like in the
   traditional approach), the advantage being that we are able to
   avoid the quotienting implied in the traditional approach. 
   In this way we get a general concept of partial signature morphism
   that is simpler and relates more directly to the concrete
   examples. 
   From this several technical benefits follow in the subsequent
   developments. 
   In the same section we also study some general properties of
   partial maps, that are relevant for our aims, such as the
   inheritance of an inclusion system and pushouts from the original
   category.  

\item We extend the inclusion system based construction of partial
  signatures morphisms to the other components of the concept of
  institution, namely the sentence and the model functors. 
  The main result of this section is that the whole construction gets
  a $\3/2$-institution. 

\item In the final section we prove some properties of the generic
  $\3/2$-institutions thus constructed that are relevant in the
  conceptual blending applications, the most important result being
  the existence of lax cocones with model amalgamation. 

\end{enumerate}

\section{A review of $\3/2$-institutions}

\subsection{Categories, monads}

In general we stick to the established category theoretic terminology
and notations, such as in 
\cite{MacLane:Categories-for-the-Working-Mathematician-1998}.
But unlike there we prefer to use the diagrammatic notation for
compositions of arrows in categories, i.e. if $f \co A \ra B$ and
$g\co B \ra C$ are arrows then $f;g$ denotes their composition. 
The domain of an arrow/morphism $f$ is denoted by $\dom{f}$ while its
codomain is denoted by $\cod{f}$. 
$\Set$ denotes the category of sets and functions and $\CAT$ the
``quasi-category'' of categories and functors.\footnote{This means it
  is bigger than a category since the hom-sets are classes rather than
  sets.}  

The \emph{dual} of a category $\C$ (obtained by formally reversing its
arrows) is denoted by $\C^\varominus$. 

Given a category $\C$, a triple $(\Delta,\delta,\mu)$ constitutes a
\emph{monad} in $\C$ when $\Delta \co \C \to \C$, and $\delta$ and
$\mu$ are natural transformations $\Delta^2 \Ra \Delta$ and $1_{\C}
\Ra \Delta$, respectively such that following diagrams commute: 
$$\xy
\xymatrix{
\Delta(\Sigma) \ar[r]^{\delta_{\Delta(\Sigma)}}
\ar[dr]_{1_{\Delta(\Sigma)}} & \Delta^2 (\Sigma) \ar[d]^{\mu_\Sigma} &
\ar[l]_{\Delta(\delta_\Sigma)} \ar[dl]^{1_{\Delta(\Sigma)}} 
\Delta(\Sigma) & \Delta^3(\Sigma) \ar[r]^{\mu_{\Delta(\Sigma)}}
  \ar[d]_{\Delta(\mu_\Sigma)} & \Delta^2 (\Sigma) \ar[d]^{\mu_\Sigma} \\  
 & \Delta(\Sigma) & & \Delta^2 (\Sigma) \ar[r]_{\mu_\Sigma} &
 \Delta(\Sigma) 
}
\endxy$$
The \emph{Kleisli category} $\C_\Delta$ of the monad
$(\Delta,\delta,\mu)$ has the objects of $\C$ but an arrow
$\theta_\Delta \co A \to B$
in $\C_\Delta$ is an arrow $\theta \co A \to \Delta(B)$ in $\C$.  
The composition in $\C_\Delta$ is defined as shown below: 
$$\xy
\xymatrix{
A \ar[d]_{\theta_\Delta} & A \ar[d]^\theta & \\
B \ar[d]_{\theta'_\Delta} & \Delta(B) \ar[d]^{\Delta(\theta')} & \\
C & \Delta^2 (C) \ar[r]_{\mu_{C}} & \Delta(C)
}
\endxy$$

The following functor \cite{3/2Inst} extends the well known power-set
functor from sets to categories.  
The \emph{power-set functor on categories} $\P \co \CAT \ra \CAT$ is
defined as follows:  
\begin{itemize}

\item for any category $\C$,
\begin{itemize}

\item $|\P\C| = \{ A \mid A \subseteq |C| \}$ and 
$\P \C(A,B) = 
\{ H \subseteq \C \mid \dom{h} \in A, \cod{h} \in B \text{ for each }
h \in H \}$; and

\item composition is defined by 
$H_1 ; H_2 = \{ h_1 ; h_2 \mid h_1 \in H_1, h_2 \in H_2, \cod{h_1} =
\dom{h_2}  \}$;  then $1_A = \{ 1_a \mid a\in A \}$ are the identities.

\end{itemize}

\item for any functor $F \co \C \ra \C'$,  
$\P F(A) = F(A) \subseteq |\C'|$ and $\P F(H) = F(H) \subseteq \C'$.

\end{itemize}
Moroever, like in the case of sets, this construction extends to a
\emph{monad $(\P,\{\_\},\cup)$} in $\CAT$. 
Then $\CAT_{\!\!\P}$ denotes its associated Keisli category. 

\subsection{Partial functions}

A \emph{partial function} $f \co A \pto B$ is a binary relation 
$f \subseteq A \times B$ such that $(a,b), (a,b') \in f$ implies 
$b = b'$. 
The \emph{definition domain} of $f$, denoted $\DOM(f)$ is the set 
$\{ a \in A \mid \exists b \ (a,b)\in f \}$. 
A partial function $f\co A \pto B$ is called \emph{total} when
$\DOM(f) = A$.  
We denote by $f^0$ the restriction of $f$ to $\DOM(f) \times B$; this
is a total function. 
Partial functions yield a subcategory of the category of binary
relations, denoted $\Setp$. 


\subsection{$\3/2$-categories}

A \emph{$\3/2$-category} is just a category such that its
hom-sets are partial orders, and the composition preserve these
partial orders.  
In the literature $\3/2$-categories are also called \emph{ordered
  categories} or \emph{locally ordered categories}. 
In terms of enriched category theory \cite{kelly-cat}, $\3/2$-category
are just categories enriched by the monoidal category of partially
ordered sets.
 
Given a $\3/2$-category $\C$ by $\C^\varobar$ we denote its `vertical'
dual which reverses the partial orders, and by $\C^\varoplus$ its
double dual $\C^{\varominus\varobar}$. 
Given $\3/2$-categories $\C$ and $\C'$, a \emph{strict $\3/2$-functor} 
$F \co \C \ra \C'$ is a functor $\C \ra \C'$ that preserves the
partial orders on the hom-sets.
\emph{Lax functors} relax the functoriality conditions  
$F(h);F(h') = F(h;h')$ to  $F(h);F(h') \leq F(h;h')$ 
(when $\cod{h} = \dom{h'}$) and $F(1_A) = 1_{F(A)}$  to $1_{F(A)} \leq
F(1_A)$.  
If these inequalities are reversed then $F$ is an 
\emph{oplax functor}. 
This terminology complies to \cite{borceux94} and to more recent
literature, but in earlier literature \cite{kelly-street74,jay91} this
is reversed.   
Note that oplax + lax = strict. 
In what follows whenever we say ``$\3/2$-functor'' without the
qualification ``lax'' or ``oplax'' we mean a functor which is either
lax or oplax.  

Lax functors can be composed like ordinary functors; we denote by
$\3/2 \CAT$ the category of $\3/2$-categories and lax functors. 
  
Most typical examples of a $\3/2$-category are $\Setp$ -- the category
of partial functions in which the ordering between partial functions
$A \pto B$ is given by the inclusion relation on the binary relations
$A \ra B$, and $\mathbf{P\!oS\!et}$ -- the category partial ordered sets (with
monotonic mappings as arrows) with orderings between monotonic
functions beign defined point-wise ($f \leq g$ if and only if $f(p)
\leq g(p)$ for all $p$).   

Let us consider the power-set monad on categories as defined above.
Given the partial order on each $\P \C$
given by category inclusions, the Kleisli category $\CAT_{\!\!\P}$ admits a
two-fold refinement to a $\3/2$-category:
\begin{enumerate}

\item morphisms $\C \ra \P \C'$ are allowed to be lax functors rather
  than (strict) functors, and  

\item we consider the point-wise partial order on the class of the lax
  functors  $\C \ra \P \C'$ that is induced by the partial order on
  $\P \C'$.  

\end{enumerate}
Let us denote the $\3/2$-category thus obtained by $\3/2 (\CAT_{\!\!\P})$. 

Unlike in the case of ordinary categories, colimits in
$\3/2$-categories come in several different flavours according to the
role played by the order on the arrows. 
Here we recall some of these for the particular emblematic case of
pushouts; the extension to other types of colimits being obvious.  

Given a span $\varphi_1,\varphi_2$ of arrows in a $\3/2$-category, a
\emph{lax cocone} for the span consists of arrows
$\theta_0,\theta_1,\theta_2$ such that there are inequalities as shown in
the following diagram:
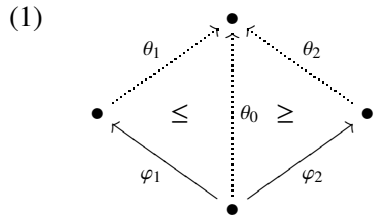
\begin{equation}\label{lax-cocone-equation}
\xymatrix @C-1.5em{
 & & \bullet & & \\
\bullet \ar@{.>}[urr]^{\theta_1} & { \ \ \ \ \leq} & &
{ \geq \ \ \ \ } &
\bullet \ar@{.>}[ull]_{\theta_2}\\
 &  & \bullet \ar[ull]^{\varphi_1} \ar[urr]_{\varphi_2} \ar@{.>}[uu]_{\theta_0}& &
}
\end{equation}
When the two inequalities are both equalities, this is a \emph{strict}
cocone. 
In this case $\theta_0$ is redundant and the data collapses to the
equality $\varphi_1 ; \theta_1 = \varphi_2 ; \theta_2$. 

A lax cocone like in diagram \eqref{lax-cocone-equation} is:
\begin{itemize}

\item \emph{pushout} when it is strict and for any strict cocone
  $\theta'_1, \theta'_2$ there exists and unique arrow $\mu$ that is
  mediating, i.e. $\theta_k ; \mu = \theta'_k$, $k = 1,2$; 

\item \emph{lax pushout} when for any lax cocone   
$\theta'_0, \theta'_1, \theta'_2$ there exists an unique mediating
arrow $\mu$, i.e. $\theta_k ; \mu = \theta'_k$, $k = 0,1,2$;

\item \emph{weak (lax) pushout} when the uniqueness condition on the
  mediating arrow is dropped from the above properties; 

\item \emph{near pushout} when for any lax cocone $\theta'_0,
  \theta'_1, \theta'_2$ the set of mediating arrows $\{ \mu \mid
  \theta_k ; \mu \leq \theta'_k, k = 0,1,2 \}$ has a maximal element. 

\end{itemize}
Pushouts are not a proper $\3/2$-categorical concept because they do
not involve in any way the orders on the arrows. \vsp

Lax pushouts represents the instance of a natural concept of colimit
from general enriched category theory \cite{kelly-cat} to
$\3/2$-categories; however in concrete situations, unlike their
cousins from ordinary category theory, they can be very
difficult to grasp and sometimes appearing quite inadequate.
For example in $\Setp$, if 
$\DOM \varphi_1 \cap \DOM \varphi_2 \not= \emptyset$ then the span
$(\varphi_1,\varphi_2)$ does not have a lax pushout. 
This is caused by the discrepancy between a lot of laxity at the
level of diagrams and of the arrows on the one hand (allowing for
unbalanced cocones in which low components may coexist with high
components), and the strictness required in the universal property on
the other hand.  
A remedy for this, that was proposed in \cite{3/2Inst}, is to
restrict the cocones to designated subclasses of arrows as follows.

Given a (1-)subcategory $\T \subseteq \C$ of a $\3/2$-category $\C$, a
\emph{lax $\T$-cocone} for a span $(\varphi_1,\varphi_2)$ is a lax
cocone $(\theta_0,\theta_1,\theta_2)$ for the span such that
$\theta_k\in \T$, $k=0,1,2$.
A \emph{lax $\T$-pushout} is a minimal lax $\T$-cocone, i.e. for any
lax $\T$-cocone  $(\theta'_0,\theta'_1,\theta'_2)$ there exists an unique
mediating arrow $\mu\in \T$ such that $\theta_k ; \mu = \theta'_k$,
$k=0,1,2$. 

This definition extends in the obvious way to general colimits and to
the weak case (by dropping off the requirement on the uniqueness of
$\mu$).  

For example, in $\Setp$ by letting $\T$ \ be the class of total
functions, \emph{any} span of partial functions admits a lax
$\T$-pushout.  \vsp


\subsection{Institutions}

An  \emph{institution} $\I = 
(\Sign^{\I}, \Sen^{\I}, \Mod^{\I}, \models^{\I})$ consists of 
\begin{itemize}

\item a category $\Sign^{\I}$ whose objects are called
  \emph{signatures},

\item a sentence functor $\Sen^{\I} \co \Sign^{\I} \ra \Set$
  defining for each signature a set whose elements are called
  \emph{sentences} over that signature and defining for each signature
  morphism a \emph{sentence translation} function, 

\item a model functor $\Mod^{\I} \co (\Sign^{\I})^{\varominus} \ra \CAT$
  defining for each signature $\Sigma$ the category
  $\Mod^{\I}(\Sigma)$ of \emph{$\Sigma$-models} and $\Sigma$-model
  homomorphisms, and for each signature morphism $\varphi$ the
  \emph{reduct} functor $\Mod^{\I}(\varphi)$,  

\item for every signature $\Sigma$, a binary 
  \emph{$\Sigma$-satisfaction relation}
  $\models^{\I}_{\Sigma} \subseteq |\Mod^{\I} (\Sigma)|
  \times \Sen^{\I} (\Sigma)$, 

\end{itemize}
such that for each morphism 
$\varphi$,  the \emph{Satisfaction Condition}
\begin{equation}
M'\models^{\I}_{\Sigma'} \Sen^{\I}(\varphi)\rho \text{ if and only if  }
\Mod^{\I}(\varphi)M' \models^{\I}_\Sigma \rho
\end{equation}
holds for each $M'\in |\Mod^{\I} (\cod{\varphi})|$ and $\rho \in
\Sen^{\I} (\dom{\varphi})$. 
\[
\xymatrix{
    \dom{\varphi} \ar[d]_{\varphi} & |\Mod^{\I}(\dom{\varphi})|
    \ar@{-}[r]^-{\models^{\I}_{\dom{\varphi}}} & 
    \Sen^{\I}(\dom{\varphi}) \ar[d]^{\Sen^{\I}(\varphi)} \\
    \cod{\varphi} & | \Mod^{\I}(\cod{\varphi})| \ar[u]^{\Mod^{\I}(\varphi)} 
    \ar@{-}[r]_-{\models^{\I}_{\cod{\varphi}}} & \Sen^{\I}(\cod{\varphi})
  }
\] 
We may omit the superscripts or subscripts from the notations of the
components of institutions when there is no risk of ambiguity. 
For example, if the considered institution and signature are clear,
we may denote $\models^{\I}_\Sigma$ just by $\models$. 
For $M = \Mod(\varphi)M'$, we say that $M$ is the
\emph{$\varphi$-reduct} of $M'$.

\begin{example}[Propositional logic -- $\PL$]
\begin{rm}
This is defined as follows.
$\Sign^{\PL} = \Set$, and for any set $P$, $\Sen(P)$ is generated by
the grammar 
\[
S ::= P \mid S \wedge S \mid \neg S 
\]
and $\Mod^{\PL}(P) = (2^P,\subseteq)$.
For any $M \in |\Mod^{\PL} (P)|$, depending on convenience, we may
consider it either as a subset $M\subseteq P$ or equivalently as a
function $M \co P \ra 2 = \{ 0, 1 \}$.

For any function $\varphi \co P \ra P'$, $\Sen^{\PL}(\varphi)$
replaces the each element $p\in P$ that occurs in a sentence $\rho$ by 
$\varphi(p)$, and $\Mod^{\PL}(\varphi)(M') = \varphi;M$ for each
$M'\in 2^{P'}$. 
For any $P$-model $M \subseteq P$ and $\rho\in \Sen^{\PL}(P)$,
$M\models\rho$ is defined by induction on the structure of $\rho$ by
$(M \models p) = (p\in M)$, 
$(M \models \rho_1 \wedge \rho_2) = 
(M \models \rho_1) \wedge (M \models \rho_2)$ and 
$(M \models \neg\rho) = \neg(M \models \rho)$.  
\end{rm}
\end{example}

\begin{example}[Many-sorted algebra -- $\MSA$]
\begin{rm}
  The \emph{$\MSA$-signatures} are pairs $(S,F)$ consisting of a set $S$ of sort
  symbols and of a family $F = \{ F_{w\ra s} \mid w\in S^*, s\in S \}$ of
  sets of function symbols indexed by arities (for the arguments) and sorts (for
  the results).\footnote{By $S^*$ we denote the set of strings of sort symbols.}
  \emph{Signature morphisms} $\varphi \co (S,F) \ra (S',F')$ consist of a
  function $\varphi^\st \co S \ra S'$ and a family of functions 
  $\varphi^\op = 
  \{ \varphi^\op_{w\ra s} \co F_{w\ra s} \ra F'_{\varphi^\st (w) \ra
    \varphi^\st (s)} \mid w\in S^*, s\in S \}$.

  The \emph{$(S,F)$-models} $M$, called algebras, interpret each sort symbol $s$
  as a set $M_s$ and each function symbol $\sigma\in F_{w\ra s}$ as a function
  $M_\sigma$ from the product $M_w$ of the interpretations of the argument sorts
  to the interpretation $M_s$ of the result sort.
  An \emph{$(S,F)$-model homomorphism} $h \co M \ra M'$ is an indexed family of
  functions $\{ h_s \co M_s \ra M'_s \mid s\in S \}$ such that 
  $h_s (M_\sigma (m)) = M'_\sigma (h_w (m))$ for each $\sigma \in F_{w \ra s}$
  and each $m\in M_w$, where $h_w \co M_w \ra M'_w$ is the canonical
  componentwise extension of $h$, i.e.\ 
  $h_w (m_1, \dots, m_n) = (h_{s_1}(m_1), \dots, h_{s_n}(m_n))$ for 
  $w = s_1 \dots s_n$ and $m_i \in M_{s_i}$.

  For each signature morphism $\varphi \co (S, F) \ra (S', F')$, the
  \emph{reduct} $\Mod(\varphi)(M')$ of an $(S', F')$-model $M'$ is defined by
  $\Mod(\varphi)(M')_x = M'_{\varphi(x)}$ for each sort or function symbol $x$
  from the domain signature of $\varphi$.

  For each signature $(S, F)$, $T_{(S,F)} = ((T_{(S,F)})_s )_{s \in S}$
  is the least family of sets such that $\sigma(t) \in (T_{(S,F)})_s$ for all
  $\sigma \in F_{w \ra s}$ and all tuples $t \in (T_{(S,F)})_{w}$.  The elements
  of $(T_{(S,F)})_s$ are called \emph{$(S,F)$-terms of sort $s$}.  For each
  $(S,F)$-algebra $M$, the \emph{evaluation of an $(S,F)$-term $\sigma(t)$ in
    $M$}, denoted $M_{\sigma(t)}$, is defined as $M_\sigma (M_t)$, where $M_t$
  is the componentwise evaluation of the tuple of $(S,F)$-terms $t$ in $M$.

  \emph{Sentences} are the usual first order sentences built from equational
  atoms $t=t'$, with $t$ and $t'$ (well-formed) terms of the same sort, by
  iterative application of Boolean connectives ($\wedge$, $\impl$, $\neg$,
  $\vee$) and quantifiers ($\forall X$, $\exists X$ -- where $X$ is a sorted
  set of variables).
  Sentence translations along signature morphisms just rename the sort and
  function symbols according to the respective signature morphisms.
  They can be formally defined by recursion on the structure of the sentences.
  The satisfaction of sentences by models is the usual Tarskian satisfaction
  defined recursively on the structure of the sentences.  (As a special note for
  the satisfaction of the quantified sentences, defined in this formalisation by
  means of model reducts, we recall that
  $M \models_\Sigma (\forall X)\rho$ if and only if $M' \models_{\Sigma+X} \rho$
  for each expansion $M'$ of $M$ to the signature $\Sigma+X$ that adds the
  variables $X$ as new constants to $\Sigma$.)
\end{rm}
\end{example}

Theories and theory morphisms are one of the crucial concepts
in institution theory and its applications to formal specification.
Traditionally theories model logic-based formal specifications, while
theory morphisms model relations between specification modules, such
as imports, renaming, views, etc. (see
\cite{iimt,sannella-tarlecki-book}). 
In any institution, a \emph{theory} is a pair $(\Sigma,E)$ consisting
of a signature $\Sigma$ and a set $E$ of $\Sigma$-sentences.  
A \emph{theory morphism} $\varphi \co (\Sigma,E) \ra (\Sigma',E')$ 
is a signature morphism $\varphi \co \Sigma \ra \Sigma'$
such that $E' \models \Sen(\varphi)E$.
It is easy to check that the theory morphisms are closed under the
composition given by the composition of the signature morphisms; this
gives the \emph{category of the theories of $\I$} denoted $\Th^\I$. 

Given a theory $(\Sigma,E)$ its \emph{closure} is $(\Sigma,E^\bullet)$
where $E^\bullet = \{ e\in \Sen(\Sigma) \mid E  \models_\Sigma e \}$. 

\subsection{$\3/2$-institutions}

According to \cite{3/2Inst}, a \emph{$\3/2$-institution} 
$\I = ( \Sign^{\I}, \Sen^{\I}, \Mod^{\I}, (\models^{\I}_\Sigma)_{\Sigma \in
  |\Sign^{\I}|} )$ consists of 
\begin{itemize}

\item a $\3/2$-category of signatures $\Sign^{\I}$, 

\item an $\3/2$-functor $\Sen^{\I} \co \Sign^{\I} \ra
  \Setp$, called the \emph{sentence functor}, 

\item an lax $\3/2$-functor  
$\Mod^{\I}\co(\Sign^{\I})^{\varoplus}\rightarrow \3/2 (\CAT_{\!\!\P})$, called the
\emph{model functor}, 

\item for each signature $\Sigma\in |\Sign^\I|$ a 
\emph{satisfaction relation} 
$\models_\Sigma^{\I} \ \subseteq \ |\Mod^{\I}(\Sigma)|\times \Sen^{\I}(\Sigma)$ 

\end{itemize}
such that for each morphism 
$\varphi \in \Sign^{\I}$, 
the \emph{Satisfaction Condition}
\begin{equation}\label{sat-cond-eq}
M'\models^{\I}_{\cod{\varphi}} \Sen^{\I}(\varphi)\rho
\ \ \text{ if and only if } \ \
M \models^{\I}_{\dom{\varphi}} \rho 
\end{equation}
holds for each $M'\in |\Mod^{\I} (\cod{\varphi})|$, 
$M \in |\Mod^{\I}(\varphi)M'|$ and 
$\rho \in \DOM (\Sen^{\I} (\varphi))$.


The difference between $\3/2$-institutions and ordinary institutions
(also called \emph{1-institutions}) is determined by the 
$\3/2$-categorical structure of the signature morphisms which
propagates to the sentence and to the model functors.   
Consequently the Satisfaction Condition \eqref{sat-cond-eq} takes an
appropriate format. 
Thus, for each signature morphism $\varphi$
its corresponding sentence translation $\Sen(\varphi)$ is a partial
function $\Sen(\dom{\varphi}) \pto \Sen(\cod{\varphi})$ and moreover
whenever $\varphi \leq \theta$ we have that 
$\Sen(\varphi) \subseteq \Sen(\theta)$. 
The sentence functor $\Sen$ can be either lax or oplax; depending on
how is this we may call the respective $\3/2$-institution as
\emph{lax} or \emph{oplax $\3/2$-institution}. 
In many concrete situations it happens that $\Sen$ is strict while
some general results require it to be either lax or oplax or both. 

The model reduct $\Mod(\varphi)$ is an \emph{lax} functor
$\Mod(\cod{\varphi}) \ra \P \Mod(\dom{\varphi})$ meaning that for each
$\Sigma'$-model we have a \emph{set of reducts} rather than a single 
reduct. 
In concrete examples this is a direct consequence of the partiality
of $\varphi$: in the reducts the interpretation of the symbols on
which $\varphi$ is not defined is unconstrained, therefore there may
be many possibilities for their interpretations. 
``Many'' here includes also the case when there is no interpretation.



\begin{itemize}

\item[--] 
The fact that $\Mod$ is a $\3/2$-functor implies also that whenever
$\varphi \leq \theta$ we have $\Mod(\theta) \leq \Mod(\varphi)$,
i.e. $\Mod(\theta)M' \subseteq \Mod(\varphi)M'$, etc. 

\item[--]
The lax aspect of $\Mod$ means that for signature morphisms
$\varphi$ and $\varphi'$ such that $\cod{\varphi} = \dom{\varphi'}$ and
for any  $\cod{\varphi'}$-model $M''$, we have that 
\[
\Mod(\varphi)(\Mod(\varphi')M'') \subseteq
\Mod(\varphi;\varphi')M''  
\]
and for each signature $\Sigma$ and for each $\Sigma$-model $M$ that 
\[
M \in \Mod(1_\Sigma)M. 
\]

\item[--]
The lax aspect of the reduct functors $\Mod(\varphi)$ means that for
model homomorphisms $h_1, h_2$ such that $\cod{h_1}=\dom{h_2}$ we have
that 
\[
\Mod(\varphi)(h_1);\Mod(\varphi)(h_2) \subseteq
\Mod(\varphi)(h_1;h_2)
\] 
and for each $M'\in \Mod(\cod{\varphi})$ and each $M\in
\Mod(\varphi)M'$ that 
\[
1_M \in \Mod(\varphi)1_{M'}.
\]
\end{itemize}
As already mentioned above model homomorphisms do not play yet any
role in conceptual blending or in other envisaged applications of
$\3/2$-institutions.
Hence the lax aspect of model functors is for the moment a purely
theoretical feature which is however supported naturally by all
examples. \vsp

In \cite{vidal-tur2010} there is a 2-categorical generalization of the
concept of institution, called \emph{2-institution}, that consider
$\Sign$ to be a 2-category, $\Sen\co \Sign \to \CAT$ and $\Mod\co
\Sign^\varominus \to \CAT$ to be pseudo-functors,
and that takes a (quite sophisticated categorically) many-valued
approach to the satisfaction relation.   
From these we can see immediately that $2$-institutions of
\cite{vidal-tur2010} do not cover the concept of $\3/2$-institution
through the perspective of $\3/2$-categories as special cases of
$2$-categories, the functors $\Sen$ and $\Mod$ in $2$-institutions
diverging from those in $\3/2$-institutions in two ways: they are
pseudo-functors (in $\3/2$-category theory this means just ordinary
functors) and their targets do not match those of
$\3/2$-institutions. 
This lack of convergence is due to the two extensions aiming to
different application domains. 







\subsection{$\3/2$-institutions: examples}

The examples given in this section are imported from \cite{3/2Inst}. 

The following expected example shows that the concept of
$\3/2$-institution constitute a generalisation of the concept of
institution. 

\begin{example}[Institutions]
\begin{rm}
Each 1-institution can be trivially regarded as a $\3/2$-institution
by regarding its category of signatures as a $\3/2$-category with
discrete partial orders. 
\end{rm}
\end{example}

\begin{example}[Propositional logic with partial morphisms of
  signatures -- $\3/2 \PL$]\label{3/2-pl-ex}
\begin{rm}
This example extends the ordinary institution $\PL$ to a
$\3/2$-institution by considering partial functions rather than total
functions as signature morphisms; thus $\Sign = \Setp$. 
\vsp

SENTENCES.
While for each set $P$, $\Sen(P)$ is like in $\PL$, for any partial
function $\varphi \co P \pto P'$ the sentence translation
$\Sen(\varphi)$ translates like in $\PL$ but only the sentences
containing only propositional variables $P$ that are translated
by $\varphi$, i.e. that belong to $\DOM \varphi$; hence the
partiality of $\Sen(\varphi)$.  
More precisely we have that  
$\DOM (\Sen\varphi) = \Sen^{\PL} (\DOM \ \varphi)$ and for each 
$\rho\in \DOM (\Sen\varphi)$ we have that $\Sen(\varphi)\rho =
\Sen^{\PL} (\varphi^0)\rho$ .  
The sentence functor is a \emph{strict} $\3/2$-functor

MODELS.
The $\3/2 \PL$ models and model homomorphisms are those of $\PL$,
but their reducts differ from those in $\PL$. 
Given a partial function $\varphi \co P \pto P'$ and a $P'$-model 
$M' \co P' \ra 2$, 
\[
\Mod(\varphi)M' = \{ M \co P \to 2 \mid M_p =
M'_{\varphi^0(p)} \text{ for all }p\in \DOM \ \varphi \}.
\]
On the model homomorphisms the reduct is defined by 
\[
\Mod(\varphi)(M' \subseteq N') = 
\{ M \subseteq N \mid M \in \Mod(\varphi)M', N \in \Mod(\varphi)N' \}.
\]

SATISFACTION. \
The satisfaction relation of $\3/2 \PL$ is inherited from $\PL$.
\end{rm}
\end{example}

\begin{example}[Many sorted algebra with partial morphisms of
  signatures -- $\3/2 \MSA$]\label{3/2-msa-ex}
\begin{rm}
In this example we extend the $\MSA$ institution to its $\3/2$ variant
in a way that parallels the extension of $\PL$ to $\3/2 \PL$. 
For this reason we will give only the definitions and rather skip the
arguments. 

Given $\MSA$ signatures, a \emph{partial $\MSA$-signatures morphism}
$\varphi \co (S,F) \pto (S',F')$ consists of 
\begin{itemize}

\item a partial function $\varphi^\st \co S \pto S'$, and 

\item for each $w \in (\DOM \varphi^\st)^*$ and $s \in \DOM
  \varphi^\st$ a partial function 
  $\varphi^\op_{w\ra s} \co F_{w\ra s} \pto F'_{\varphi^\st w\ra
    \varphi^\st s}$.

\end{itemize}
Given $\varphi \co (S,F) \pto (S',F')$ and $\varphi' \co (S',F') \pto
(S'',F'')$ their composition $\varphi;\varphi'$ is defined by 
\begin{itemize}

\item $(\varphi;\varphi')^\st = \varphi^\st ; \varphi'^\st$, and 

\item for each  $w \in (\DOM (\varphi;\varphi')^\st)^*$ and 
$s \in \DOM (\varphi;\varphi')^\st$: \  
$(\varphi;\varphi')^\op_{w\ra s} = 
\varphi^\op_{w\ra s} ; \varphi'^\op_{\varphi^\st w\ra \varphi^\st s}$. 

\end{itemize}
Given $\varphi,\theta \co (S,F) \pto (S',F')$, then $\varphi \leq
\theta$ if and only if 
\begin{itemize}

\item $\varphi^\st \subseteq \theta^\st$, and 

\item for each $w \in (\DOM \varphi^\st)^*$ and $s \in \DOM
  \varphi^\st$: \ $\varphi^\op_{w\ra s} \subseteq \theta^\op_{w\ra s}$.

\end{itemize}
Under these definitions the partial $\MSA$-signature morphisms form a
$\3/2$-category, which is the category of the $\3/2 \MSA$
signatures. 

Given a partial $\MSA$-signature morphism $\varphi$ we denote by $\DOM
\varphi$ the signature $(\DOM \varphi^\st, \DOM \varphi^\op)$ where  
$(\DOM \varphi^\op)_{w \to s} = \DOM \varphi^\op_{w\to s}$ and by
$\varphi^0 \co \DOM \varphi \to \cod{\varphi}$ the resulting (total)
$\MSA$-signature morphism.

For any signature $\Sigma$, 
$\Sen^{\3/2 \MSA}(\Sigma) = \Sen^{\MSA}(\Sigma)$ and for any partial
$\MSA$-signature morphism $\varphi$, $\Sen^{\3/2 \MSA}(\varphi)$ is
defined by 
\begin{itemize}

\item $\DOM \ \Sen^{\3/2 \MSA}(\varphi) = \Sen^\MSA (\DOM \varphi)$
and 

\item for each sentence $\rho \in \DOM \ \Sen^{\3/2 \MSA}(\varphi)$, 
$\Sen^{\3/2 \MSA}(\varphi)\rho = \Sen^\MSA (\varphi^0)\rho$. 

\end{itemize}
Like for $\3/2 \PL$ this yields also a \emph{strict} $\3/2$-functor. 
For any signature $\Sigma$, 
$\Mod^{\3/2 \MSA}(\Sigma) = \Mod^{\MSA}(\Sigma)$ and for any partial
$\MSA$-signature morphism $\varphi$, each $\cod{\varphi}$-model $M'$,
$\Mod^{\3/2 \MSA}(\varphi)M' = M$ is 
defined by 
\begin{itemize}

\item for each sort symbol $s$ in $\DOM \varphi$, $M_s = M'_{\varphi^\st s}$,
  and  

\item for each operation symbol $\sigma$ in $\DOM \varphi$, 
$M_\sigma = M'_{\varphi^\op \sigma}$.  

\end{itemize}
The definition on model homomorphisms is similar, we skip it here. 
Under these definitions, $\Mod^{\3/2 \MSA}$ is a lax functor. 

The satisfaction relation is inherited from $\MSA$, and the argument
for the Satisfaction Condition in $\3/2 \MSA$ is similar to that in
$\3/2 \PL$.  
\end{rm}
\end{example}

\begin{example}\label{3/2-submsa-ex}
\begin{rm}
The $\3/2 \MSA$ example can be twisted by considering less partiality
in the signature morphisms.
This can be done in several ways, in each case a different
$\3/2$-`sub-institution' of $\3/2 \MSA$ emerges. 
\begin{enumerate}

\item We constrain $\varphi^\st$ to be total functions.

\item We let $\varphi^\st$ to be partial functions but we constrain
  $\varphi^\op_{w\to s}$ to be total. 

\end{enumerate}
\end{rm}
\end{example}

\begin{example}\label{3/2-views-ex}
\begin{rm}
The pattern of Ex.~\ref{3/2-msa-ex} can be applied to the
extension of $\MSA$ that takes the `first order views' of \cite{Views}
in the role of signature morphisms.
Since first order views are more general the the $\MSA$ signature
morphisms, the resulting $\3/2$-institution based upon partial first
order views can thought as an extension of $\3/2 \MSA$.   
\end{rm}
\end{example}

So far the Examples
\ref{3/2-pl-ex}, \ref{3/2-msa-ex}, \ref{3/2-submsa-ex} and
\ref{3/2-views-ex} are based upon a pattern that can be described as
follows: 
\begin{enumerate}

\item Consider a concrete 1-institution (that may be quite common). 

\item Consider some form of partiality for its signature morphisms;
  often this can be done in several different ways (see
  Ex.~\ref{3/2-submsa-ex}).  

\item Keep the sentences and the models of the original institution,
  but based on the partiality of the signature morphisms extend the
  concepts of sentence translations and of model reducts to
  $\3/2$-institutional ones.
  The partiality of the sentence translations amounts to the fact that
  only the sentences that only involve symbols from the definition domain
  of the (partial) signature morphism can be translated.
  The relation-like aspect of the model reducts amounts to the fact
  that symbols that are outside the definition domain of the (partial)
  signature morphisms can be interpreted in several different ways in
  the models.  

\item The satisfaction relation of the resulting $\3/2$-institution is
  inherited from the original 1-institution. 

\end{enumerate}
This pattern pervades a lot of useful $\3/2$-institutions and can be
captured as a generic mathematical construction that derives
$\3/2$-institutions from 1-institutions.
The main topic of this paper is precisely to explain mathematically
this pattern, and then on such basis to derive general properties that
are useful in the envisaged applications of $\3/2$-institution
theory. 
However in \cite{3/2Inst} there are interesting examples of
$\3/2$-institutions that fall short off this pattern.

\section{Generic partial signature morphisms} 
\label{gen-3/2-institution-sec}

In this section we present a generic method for constructing
$\3/2$-institutions on top of 1-institutions that is based on extending
the category of the signatures by considering partiality for the
signature morphisms.   
Instances of this generic construction include $\3/2 \PL$, $\3/2 \MSA$
but also the $\3/2$-subinstitutions of $\3/2 \MSA$ from
Ex.~\ref{3/2-submsa-ex}. 

The structure of the section is as follows:
\begin{enumerate}

\item We recall from the literature the concept of \emph{inclusion
    system} that we employ for building generic partiality for the
  signature morphisms.  

\item Given a 1-category of $\Sign$ endowed with an inclusion system
  we build a $\3/2$-category $p\Sign$, that extends $\Sign$, and whose
  arrows are `partial maps' in $\Sign$. 
  The categorical literature has an established approach to those via
  spans (e.g. \cite{robinson-rosolini88,jay91}, etc.), and in
  principle we follow that.  
  However the distinctive feature of our approach is the use of
  inclusion systems, which leads to somehow simpler constructions
  and proofs as it avoids the quotienting inherent in the standard
  span-based approaches to partial maps. 
  We show how inclusion systems and colimits in $p\Sign$ are inherited
  from $\Sign$. 

\item Then the sentence and the model structures of the constructed 
  $\3/2$-institution are developed from those of the base  
  1-institution on the basis of the partial maps in $\Sign$. 
  The satisfaction relation of the $\3/2$-institution is inherited
  from the base 1-institution.  

\item We show how lax cocones of signature morphisms admitting model
  amalgamation in the constructed $\3/2$-institution can be
  obtained from cocones of signature morphisms admitting model
  amalgamation in the base 1-institution. 

\item We provide a taxonomy of theory morphisms in the constructed
  $\3/2$-institutions, that reflects various ways to achieve
  partiality for theory morphisms.   

\end{enumerate}


\subsection{Inclusion systems}

Inclusion systems were introduced in~\cite{modalg} as a categorical device
supporting an abstract general study of structuring of specification and
programming modules that is independent of any underlying logic.
They have been used in a series of general module algebra studies such
as~\cite{modalg,goguen-rosu2004,iimt} but also for developing
axiomatisability~\cite{rosu-iel,edins,iimt} and
definability~\cite{aiguier-barbier2005} results within the framework of the
so-called institution-independent model theory~\cite{iimt}.
Inclusion systems capture categorically the concept of set-theoretic inclusion
in a way reminiscent of how the rather notorious concept of factorization
system~\cite{borceux94} captures categorically the set-theoretic injections;
however, in many applications the former are more convenient than the latter.
Here we recall from the literature the basics of the theory of inclusion
systems.

The definition below can be found in the recent literature on inclusion systems
(e.g.~\cite{iimt}) and differs slightly from the original one of~\cite{modalg}.

  A pair of categories $\mpair{\I}{\mathcal{E}}$ is an \emph{inclusion
    system} for a category $\C$ if $\I$ and $\mathcal{E}$ are
  two broad subcategories of $\C$ such that
  \begin{enumerate} 

  \item $\I$ is a partial order (with the order relation denoted by
    $\subseteq$), and

  \item every arrow $f$ in $\C$ can be factored uniquely as 
    $f = e_f ; i_f$ with $e_f \in \mathcal{E}$ and $i_f \in \I$.

  \end{enumerate}
  The arrows of $\I$ are called \emph{abstract inclusions}, and  
  the arrows of $\mathcal{E}$ are called \emph{abstract surjections}.   
  The domain of the inclusion $i_f$ in the factorization of $f$ is called the
  \emph{image of $f$} and is denoted as $\mathrm{Im}(f)$ or $f(A)$ when $A$ is
  the domain of $f$.
  An inclusion $i \co A \ra B$ may also be denoted simply by $A \subseteq B$.


In~\cite{rosu-cazanescu97} it is shown that the class $\I$ of abstract
inclusions determines the class $\mathcal{E}$ of abstract surjections.
In this sense,~\cite{rosu-cazanescu97} gives an explicit equivalent definition
of inclusion systems that is based only on the class $\I$ of abstract
inclusions. 

Given categories $\C$ and $\C'$, each endowed with an inclusion
system, a functor $\C \to \C'$ is called \emph{inclusive} when
it maps abstract inclusions to abstract inclusions. 
This is the established structure-preserving mapping between inclusion
systems (see \cite{modalg,iimt,gis}, etc.).  

The literature contains many other examples of inclusion systems for
the categories of signatures and for the categories of models of
various institutions from logic or from specification theory.   
We recall here only a couple of them.  

\begin{example}[Inclusion system for $\PL$ signatures]\label{pl-inc-sys-ex}
\begin{rm}
The standard example of inclusion system is that from $\Set$, with
set theoretic inclusions in the role of the abstract inclusions and surjective
functions in the role of the abstract surjections.
\end{rm}  
\end{example}

\begin{example}[Inclusion systems for
  $\MSA$-signatures]\label{msa-inc-sys-ex}
\begin{rm}
  Besides the trivial inclusion system that can be defined in any category
  (i.e.\ identities as abstract inclusions and all arrows as abstract
  surjections) the category of $\MSA$-signatures admits also the following three
  non-trivial inclusion systems:

  \begin{center}
    \begin{tabular}{lll}
      \toprule
      inclusion system & abstract surjections & abstract inclusions \\
      & $\varphi \co (S, F) \ra (S', F')$ & $(S,F) \subseteq (S',F')$ \\
      \midrule
      \emph{closed} & $\varphi^\st \co S \ra S'$ surjective & $S \subseteq S'$ \\
      & & $F_{w \ra s} = F'_{w \ra s}$ for $w \in S^*$, $s\in S$
      \\  [1ex] \hdashline 
      \emph{strong} & $\varphi^\st \co S \ra S'$ surjective & $S \subseteq S'$ \\
      & $F'_{w' \ra s'} = \bigcup_{\varphi^{\st} (ws) = w's'} \varphi^{\op} (F_{w \ra s})$ &
      $F_{w \ra s} \subseteq F'_{w \ra s}$ for $w \in S^*$, $s \in S$
      \\[1ex] \hdashline
      \emph{nearly strong} & $\varphi^\st \co S \ra S'$ any function & 
      $S = S'$ \\
      & $F'_{w' \ra s'} = \bigcup_{\varphi^{\st} (ws) = w's'} \varphi^{\op} (F_{w \ra s})$ &
      $F_{w \ra s} \subseteq F'_{w \ra s}$ for $w \in S^*$, $s \in S$ \\
      \bottomrule
    \end{tabular} 
  \end{center}
\end{rm}
\end{example}

\begin{example}[Inclusion systems for theory morphisms]
\begin{rm}
In any institution such that its category $\Sign$ of signatures is
endowed with an inclusion system such that $\Sen$ is inclusive, its
category of \emph{closed} theories (which is the corresponding full
subcategory of $\Th^{\I}$) may inherit this inclusion system
in two different ways.  
This is well known in the literature (e.g. \cite{iimt}\footnote{But
  there ``theories'' are our ``closed theories''}) and goes as
shown in the following table:

  \begin{center}
    \begin{tabular}{lll}
      \toprule
      inclusion system & abstract surjections & abstract inclusions \\
      & $\varphi \co (\Sigma,E) \to (\Sigma',E')$ 
      & $(\Sigma,E) \subseteq (\Sigma',E')$ \\
      \midrule
      \emph{closed} & $\varphi \co \Sigma \to \Sigma'$ abstract
                      surjection & $\Sigma \subseteq \Sigma'$ and 
                                   $E = \Sen(\Sigma) \cap E'$\\
      [1ex] \hdashline 
      \emph{strong} & $\varphi \co \Sigma \to \Sigma'$ abstract
                      surjection and $\Sen(\varphi)E = E'   $
                    & $\Sigma \subseteq \Sigma'$ \\
      \bottomrule
    \end{tabular} 
  \end{center}

\end{rm}
\end{example}

\begin{definition}
In any category endowed with an inclusion system, a cospan of arrows
$f_1 \co A_1 \to A, f_2 \co A_2 \to A$ is called \emph{semi-inclusive}
when one of $f_1$ or $f_2$ is an abstract inclusion. 
\end{definition}

The following property of inclusion systems, which can be found in 
\cite{iimt}, has a special relevance in what follows.   

\begin{lemma}\label{pullback-lem}
In a category endowed with an inclusion system and which has
pullbacks of semi-inclusive cospans, for any $f \co A \ra B$ and any
inclusion $B' \subseteq B$ there exists an unique pullback such that
$A'\subseteq A$:  
\begin{equation}\label{pullback-diag}
\xymatrix{
A \ar[r]^f & B \\
A' \ar[u]^{\subseteq} \ar[r]_{f'} & B' \ar[u]_{\subseteq}
}
\end{equation}
\end{lemma}

\subsection{Partial signature morphisms}

Partial maps in abstract categories are well known in the literature,
one of the earliest references being \cite{robinson-rosolini88}.
There are only slight differences between different approaches, all of
them defining partial maps as equivalences classes of spans of arrows.
Here we come up with an inclusion systems-based variant that avoids
quotients. 

\begin{definition}\label{partial-def}
Given a category $\Sign$ endowed with an inclusion
system and which has pullbacks of semi-inclusive cospans, for any
$\Sigma,\Sigma' \in |\Sign|$,  
a \emph{partial $\Sign$-morphism} $\varphi \co \Sigma \pto \Sigma'$
consists of a $\Sign$-morphism $\varphi^0 \co \Sigma_0 \ra \Sigma'$
such that $\Sigma_0 \subseteq \Sigma$. 
We may denote $\Sigma_0$ by $\DOM \varphi$. 

Given $\varphi \co \Sigma \pto \Sigma'$ and $\varphi' \co \Sigma' \pto
\Sigma''$ their \emph{composition} $\varphi;\varphi'$ is defined by
the following diagram:
\begin{equation}\label{diag-comp}
\xymatrix  @C-1em {
\Sigma  \ar@{.>}@/^{2pc}/[rrrr]^-{\varphi;\varphi'} 
\ar@{.>}[rr]^\varphi & & \Sigma' \ar@{.>}[rr]^{\varphi'} & & \Sigma'' \\
 & \DOM \varphi \ar[ul]^{\subseteq} \ar[ur]_{\varphi^0} &
 (\Diamond) & \DOM
 \varphi' \ar[ur]_{\varphi'^0}
 \ar[ul]^{\subseteq}   & \\
 & & \DOM \varphi;\varphi' \ar[ul]^{\subseteq} \ar[ur]_{(\varphi^0)'}
 \ar `r[uurr] [uurr]_{(\varphi;\varphi')^0} & & 
}
\end{equation}
where the square ($\Diamond$) is the unique pullback of $\varphi^0$
and $\DOM \varphi' \subseteq \Sigma'$ given by Lemma
\ref{pullback-lem}. 

Given $\varphi,\theta \co \Sigma \pto \Sigma'$, then $\varphi \leq
\theta$ if and only if $\DOM \varphi \subseteq \DOM \theta$ and 
$\varphi^0 = (\DOM \varphi \subseteq \DOM \theta);\theta^0$. 
\[
\xymatrix  {
\dom{\theta} = \dom{\varphi}  & \cod{\theta} = \cod{\varphi} \\
\DOM \theta \ar[u]^\subseteq \ar[ur]_{\theta_0} & \\
\DOM \varphi \ar[u]^\subseteq \ar `r[uur] [uur]_{\varphi^0} & 
}
\]
\end{definition}

Note the overloading of notations $\pto$ and $\DOM\varphi$ here with
the corresponding ones from partial functions. 
In the abstract context they are meant to suggest abstract 
partiality rather than concrete partiality. 
However in the example of $\PL$ and $\3/2 \PL$ their meanings do
coincide. 
Also giving the pair $\varphi^0$ such that $\dom{\varphi^0} \subseteq
\dom{\varphi}$ is the same with giving the span  $\DOM \varphi
\subseteq \dom{\varphi}$, $\varphi^0$ in $\Sign$ 
(the first arrow being an abstract inclusion). 

\begin{proposition}
Let $p\Sign$ have the same objects as $\Sign$ and the partial
$\Sign$-morphisms as arrows. Under the definitions given in
Dfn.~\ref{partial-def}, $p\Sign$ is a $\3/2$-category. 
\end{proposition} 

\begin{proof}
The associativity of the composition in $p\Sign$ can be determined by
chasing the following diagram and by resorting to Lemma
\ref{pullback-lem}. 
\[
\xymatrix  @C-1em{
\Sigma  
\ar@{.>}[rr]^{\varphi_1} & & \Sigma_1 \ar@{.>}[rr]^{\varphi_2} & & \Sigma_2
\ar@{.>}[rr]^{\varphi_3} & & \Sigma_3 \\
 & \DOM \varphi_1 \ar[ul]^{\subseteq} \ar[ur]_{\varphi_1^0} & (1)
  & \DOM  \varphi_2 \ar[ur]_{\varphi_2^0} \ar[ul]^{\subseteq} & (2) 
  & \DOM  \varphi_3
 \ar[ul]^{\subseteq}  \ar[ur]_{\varphi_3^0} & \\
 & & \DOM \varphi_1;\varphi_2 \ar[ul]^{\subseteq} \ar[ur]_{(\varphi_1^0)'}
& (3) & \DOM \varphi_2;\varphi_3 \ar[ul]^{\subseteq}
\ar[ur]_{(\varphi_2^0)'}  \\
& & & \DOM \varphi_1;\varphi_2;\varphi_3 \ar[ul]^{\subseteq} 
\ar[ur]_{(\varphi_1^0)''} & & & 
}
\]
The squares (1), (2), (3) are unique pullbacks as determined by Lemma
\ref{pullback-lem}.  
Then the square (2)+(3) corresponds to the square ($\Diamond$) in the diagram of 
Dfn.~\ref{partial-def} for the composition
$(\varphi_1;\varphi_2);\varphi_3$ while the square (1)+(3) corresponds
to the square ($\Diamond$) for the composition
$\varphi_1;(\varphi_2;\varphi_3)$. 

The identities of $p\Sign$ are the identities of $\Sign$ (we skip
here the straightforward proof that these are identities in
$p\Sign$ indeed).  

Now we prove the preservation of the partial orders on the
hom-sets by the composition; let us do here only one side of that, the
argument for the other side being similar.  
We consider $\varphi_1 \leq \varphi_2 \co \Sigma \pto \Sigma'$ and 
$\theta \co \Sigma' \pto \Sigma''$. 
The argument for $\varphi_1 ; \theta \leq \varphi_2 ; \theta$ is
apparent by analysing the following diagram:
\[
\xymatrix  {
\Sigma  \ar@{.>}@/^{.5pc}/[rr]^{\varphi_1}_{\leq}
\ar@{.>}@/_{.5pc}/[rr]_{\varphi_2} & & \Sigma' 
\ar@{.>}[rr]^{\theta} & & \Sigma'' \\
 & \DOM \varphi_1 \ar[ul]^{\subseteq} \ar[ur]_{\varphi_1^0} &
 & \DOM  \theta \ar[ur]_{\theta^0}
 \ar[ul]^{\subseteq}   & \\
 & \DOM \varphi_2 \ar[u]^\subseteq & \DOM \varphi_1;\theta
 \ar[ul]^{\subseteq} \ar[ur]_{(\varphi_1^0)'}   & & \\
 & & \DOM \varphi_2;\theta \ar[u]^\subseteq   \ar[ul]^{\subseteq} 
         \ar `r[uur] [uur]_-{(\varphi_2^0)'}
 & & 
}
\]
\end{proof}

\begin{fact}\label{embed-fact}
There is a canonical faithful functor $[{\_}] \co \Sign \to p\Sign$
which is the identity on the objects and such that $[\chi]^0 = \chi$
for each arrow $\chi \in \Sign$. 
\end{fact}

\begin{example}
\begin{rm}
The $\3/2$-category of the signatures of $\3/2 \PL$ is the category of
partial $\Sign^\PL$-morphisms when considering the standard inclusion
system in $\Set$. 
The  $\3/2$-category of the signatures of $\3/2 \MSA$ is the category of
partial $\Sign^\MSA$-morphisms when considering the strong inclusion
system in $\Sign^\MSA$. 
The  $\3/2$-categories of the signatures of the $\3/2$-subinstitutions    
of $\3/2 \MSA$ from Ex.~\ref{3/2-submsa-ex} arise as categories of
partial $\Sign^\MSA$-morphisms when considering the closed and nearly
strong inclusion systems in $\Sign^\MSA$.
\end{rm}
\end{example}

\subsection{Inclusion systems for partial signature morphisms}

The functor of Fact \ref{embed-fact} transfers the inclusion system
of $\Sign$ to $p\Sign$; however this is not completely trivial as the
additional following technical property is needed: 

\begin{definition}
In any category endowed with an inclusion system and with pullbacks of
semi-inclusive cospans, we say that 
\emph{abstract surjections are stable under semi-inclusive pullbacks}
when for each pullback square like in diagram  
(\ref{pullback-diag}) if $f$ is an abstract surjection then $f'$ is
an abstract surjection too. 
\end{definition}

\begin{example}
\begin{rm}
The stability property under inclusive pullbacks holds widely in
examples.  
It is not difficult to check that all four inclusion systems of
Examples \ref{pl-inc-sys-ex} and \ref{msa-inc-sys-ex} have this
property. 
Let us do it here only for the strong inclusion system for the
$\MSA$ signatures. 
Consider an inclusive pullback square with respect to the strong
inclusion system of $\MSA$ signatures like in the diagram
(\ref{pullback-diag}):   
\begin{equation}\label{inc-pullback-msa-diag}
\xymatrix{
(S,F) \ar[r]^\varphi & (S_1,F_1) \\
(S',F') \ar[u]^{\subseteq} \ar[r]_{\varphi'} & (S'_1,F'_1) \ar[u]_{\subseteq}
}
\end{equation}
such that $\varphi$ is an abstract surjection. 
We have to show that $\varphi'$ is an abstract surjection too. 

Since $\varphi^\st$ is a surjective function it follows that for each
$s'\in S'_1 \subseteq S_1$ there exists $s \in S$ such that 
$\varphi^\st (s) = s'$. 
The pullback square (\ref{inc-pullback-msa-diag}) implies that the
following is a pullback square in $\Set$ (see \cite{iimt} for a
detailed general construction of pullbacks off signature morphisms in
$\MSA$): 
\begin{equation}\label{inc-pullback-msa-sorts-diag}
\xymatrix{
S \ar[r]^{\varphi^\st} & S_1 \\
S' \ar[u]^{\subseteq} \ar[r]_{\varphi'^\st} & S'_1 \ar[u]_{\subseteq}
}
\end{equation}
which means that $S' = \{ x \in S \mid \varphi^\st (x) \in S'_1 \}$. 
Consequently $s\in S'$ and $\varphi'^\st (s) = s'$. 
Thus shows that $\varphi'^\st$ is a surjective function too. 

The remaining part of the argument is slightly more intricate. 
Let $w_1, s_1$ and $\sigma_1 \in (F'_1)_{w_1 \to  s_1}$.
Since $(S'_1, F'_1) \subseteq (S_1, F_1)$ we have that 
$\sigma_1 \in (F_1)_{w_1 \to  s_1}$. 
Since $\varphi$ is abstract surjection there exists 
$w,s$ and $\sigma\in F_{w\to s}$ such that 
$\varphi^\op (\sigma) = \sigma_1$. 
By the construction of pullbacks in $\MSA$ (see diagram
(\ref{inc-pullback-msa-sorts-diag})) we know that $w\in S'^*$ and
$s\in S'$ and that $\varphi'^{\st} (w) = w_1$ and  
$\varphi'^{\st} (s) = s_1$. 
But the construction of pullbacks of $\MSA$ signature morphisms also
gives us that  
\[
F'_{w\to s} = \{ x \in F_{w\to s} \mid \varphi^{\op} (x) \in
(F'_1)_{w_1 \to  s_1} \}.
\]
Consequently $\sigma \in F'_{w\to s}$ and $\varphi'^{\op} (\sigma) \in
(F'_1)_{w_1 \to  s_1}$, which completes the proof that $\varphi'$ is
an abstract surjection of the strong inclusion system of the $\MSA$
signature morphisms. 
\end{rm}
\end{example}

\begin{proposition}
Assuming that in $\Sign$ the abstract surjections are stable under
inclusive pullbacks, the following gives an inclusion system in
$p\Sign$:  
\begin{itemize}

\item abstract inclusions: $[i]$, where $i$ is an abstract inclusion
  in $\Sign$; and 

\item abstract surjections: $\varphi$, such that $\varphi^0$ is an
  abstract surjection in $\Sign$. 

\end{itemize}
\end{proposition}

\begin{proof}
That the abstract inclusions of $p\Sign$ form a partial order follows
from the functoriality and the faithfullness of the embedding $[{\_}]$. 
That the abstract surjections in $p\Sign$ form a subcategory follows
by inspecting the diagram (\ref{diag-comp}) and by applying the
stability property under inclusive pullbacks to the square
($\Diamond)$ and to $\varphi^0$. 
Then $(\varphi^0)'$ is abstract surjection (in $\Sign$) and
consequently $(\varphi;\varphi')^0 = (\varphi^0)';\varphi'^0$ is
abstract surjection (in $\Sign$) too. 

Any $\varphi \in p\Sign$ can be factored as shown in the following
figure (with $e_\varphi$ and $i_\varphi$ being abstract surjection and
inclusion, respectively): 
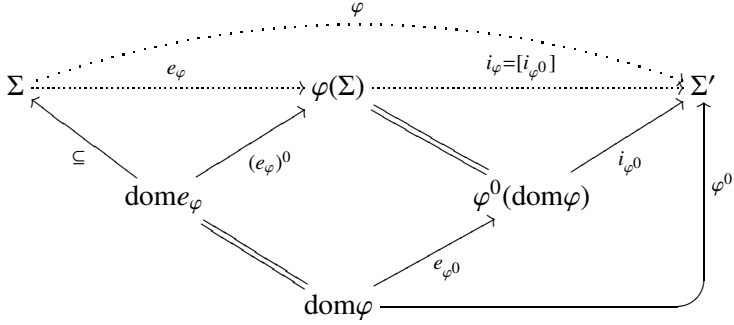
\[
\xymatrix  @C+.7em {
\Sigma  \ar@{.>}@/^{2pc}/[rrrr]^-{\varphi} 
\ar@{.>}[rr]^{e_\varphi} & & \varphi(\Sigma) 
\ar@{.>}[rr]^{i_\varphi = [i_{\varphi^0}]} & & \Sigma' \\
 & \DOM e_\varphi \ar[ul]^{\subseteq} \ar[ur]_{(e_\varphi)^0} & 
 & \varphi^0 (\DOM \varphi) \ar[ur]_{i_{\varphi^0}}
 \ar@{=}[ul]   & \\
 & & \DOM \varphi \ar@{=}[ul] \ar[ur]_{e_{\varphi^0}}
 \ar `r[uurr] [uurr]_{\varphi^0} & & 
}
\]
For showing the uniqueness of the factoring in $p\Sign$ let us assume
$\varphi = e;i$ where $e$ and $i$ are abstract surjections and
inclusions, respectively. 
There exists an abstract inclusion $i'$ in $\Sign$ such that $i =
[i']$.
It follows that $\varphi^0 = e^0 ; i'^0$.
By the uniqueness of the factoring in the inclusion system of $\Sign$
it follows that $e^0 = (e_\varphi)^0$ and that $i'^0 =
i_{\varphi^0}$, hence $e = e_\varphi$ and $i = i_\varphi$. 
\end{proof}

\begin{corollary}
The categories $p\Sign^{\PL}$ and $p\Sign^{\MSA}$ have inclusion
systems that inherit the respective inclusion systems of $\Sign^{\PL}$
(Example \ref{pl-inc-sys-ex}) and of $\Sign^{\MSA}$ (Example
\ref{msa-inc-sys-ex}). 
\end{corollary}











\subsection{Pushouts in the category of partial signature morphisms}

The following result shows that a relevant class of lax pushouts in
$p\Sign$ is determined on the basis of pushouts in $\Sign$. 
It can also be extended easily to other colimits. 

\begin{proposition}\label{pushout-partial-prop}
If $\Sign$ has (weak) pushouts then $p\Sign$ has (weak) lax
$\Sign$-pushouts.\footnote{Where $\Sign$ is considered as a
  subcategory of $p\Sign$ via the embedding of Fact \ref{embed-fact}.}  
\end{proposition}

\begin{proof}
The proof for the weak case is obtained from the proof of the non-weak
case by discarding the uniqueness properties.  
We will therefore consider here only the non-weak case. 
 
We consider a span $\varphi_k \co \Sigma_0 \to \Sigma_k$, $k=1,2$ of
partial $\Sign$-morphisms. 
Then 
\begin{enumerate}

\item (in $\Sign$) we consider pushout cocones $(\alpha_k,\chi_k)$ for
  the two spans $(\varphi^0_k,\DOM \varphi_k \subseteq \Sigma_0)$, $k=1,2$
  (see diagram \eqref{pushout-proof-diag} below); 

\item (in $\Sign$) we consider a pushout cocone $(\beta_1,\beta_2)$ for
  the span $(\alpha_1,\alpha_2)$;

\item for $k=1,2$ we define $\theta_k^0 = \chi_k; \beta_k$ and we also
  define $\theta^0_0 = \alpha_k ; \beta_k$. 

\end{enumerate}
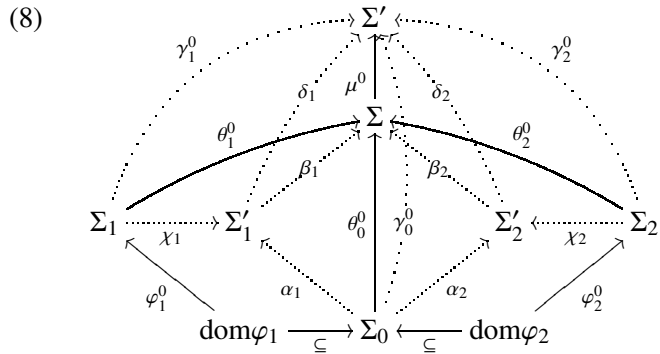
\begin{equation}\label{pushout-proof-diag}
\xymatrix{
  & & \Sigma' & & \\
  & & \Sigma \ar[u]^(.35){\mu^0} & & \\
\Sigma_1 \ar@{.>}@/^{2pc}/[uurr]^{\gamma^0_1}
\ar@/^{.5pc}/[urr]^{\theta^0_1} \ar@{.>}[r]_{\chi_1}
& \Sigma'_1 \ar@{.>}[ur]|-{\beta_1} \ar@{.>}@/^{.5pc}/[uur]|(.6){\delta_1} &  & 
  \Sigma'_2 \ar@{.>}[ul]|-{\beta_2}
  \ar@{.>}@/_{.5pc}/[uul]|(.6){\delta_2} & 
\Sigma_2  \ar@{.>}@/_{2pc}/[uull]_{\gamma^0_2}
\ar@/_{.5pc}/[ull]_{\theta^0_2} \ar@{.>}[l]^{\chi_2} \\
& \DOM\varphi_1 \ar[ul]^{\varphi^0_1} \ar[r]_-\subseteq & 
\Sigma_0 \ar@{.>}[ul]^{\alpha_1} \ar@{.>}[ur]_{\alpha_2}
\ar[uu]^{\theta^0_0}
\ar@{.>}@/_{1pc}/[uuu]|(.35){\gamma^0_0}
& \DOM\varphi_2 \ar[ur]_{\varphi^0_2} \ar[l]^-\subseteq &  
}
\end{equation}
It follows that, in $p\Sign$,
$(\theta_0 = [\theta^0_0],\theta_1 = [\theta^0_1],\theta_2 =
[\theta^0_2])$ constitutes a lax cocone  for the span
$(\varphi_1,\varphi_2)$ (see diagram \eqref{pushout-proof-diag}).  

Now we consider a lax $\Sign$-cocone
$(\gamma_0 = [\gamma^0_0],\gamma_1 = [\gamma^0_1],\gamma_2 =
[\gamma^0_2])$ for the same span.    
It follows that for $k = 1,2$, $(\gamma^0_k,\gamma^0_0)$ is a cocone
for the span $(\varphi^0_k,\DOM \varphi_k \subseteq \Sigma_0)$.
By the pushout property in $\Sign$, for $k=1,2$ there exists an unique
$\delta_k \co \Sigma'_k \to \Sigma'$ such that 
$\chi_k ; \delta_k = \gamma^0_k$ and $\alpha_k ; \delta_k =
\gamma^0_0$. 
This yields $(\delta_1,\delta_2)$ a cocone in $\Sign$ for the span
$(\alpha_1,\alpha_2)$.  
By the pushout property in $\Sign$ for the span $(\alpha_1,\alpha_2)$
there exists an unique $\mu^0 \co \Sigma \to \Sigma'$ such that for
$k=1,2$, $\beta_k ; \mu^0 = \delta_k$. 

By chasing diagram \eqref{pushout-proof-diag} we have that,
for $k=1,2$
\begin{equation}\label{pushout-proof-eq1}
\theta^0_0 ; \mu^0 = \alpha_k ; \beta_k ; \mu^0 = 
\alpha_k ;\delta_k = \gamma^0_0
\end{equation}
and 
\begin{equation}\label{pushout-proof-eq2}
\theta^0_k;\mu^0 = \chi_k ; \beta_k ; \mu^0 = 
\chi_k ; \delta_k = \gamma^0_k.
\end{equation}
Let $\mu = [\mu^0]$.
From \eqref{pushout-proof-eq1} and \eqref{pushout-proof-eq2} we obtain
that (in $p\Sign$) $\theta_k ; \mu = \gamma_k$, $k = 0,1,2$. 
The uniqueness of $\mu$ follows from the uniqueness side of the
pushout properties involved. 
\end{proof}

\section{Sentences, models and satisfaction with partial signature
  morphisms} 

In this section we complete the development of $\3/2$-institutions on
the basis of the results of the previous section. 
Therefore here the underlying technical assumption is that the
category of signatures $\Sign$ is endowed with an inclusion system
such that it has pullbacks of semi-inclusive cospans. 

\subsection{The sentence functor $p\Sen$}

The following construction represents and extension of the sentence
functor $\Sen$ of a base institution to a $\3/2$-institution
theoretic sentence functor $p\Sen$. 

\begin{definition}\label{sen-dfn}
Given an inclusive functor $\Sen \co \Sign \to \Set$, for each
partial signature $\Sign$-morphism $\varphi\in p\Sign$ we 
define a partial function $p\Sen(\varphi) \co \Sen(\dom{\varphi}) \pto 
\Sen(\cod{\varphi})$  by letting 
\begin{itemize}

\item 
$\DOM \ p\Sen(\varphi) = \Sen(\DOM \varphi)$ and 

\item for each  $\rho \in \DOM \ p\Sen(\varphi)$, 
$p\Sen(\varphi)\rho = \Sen(\varphi^0)\rho$.  

\end{itemize}
\end{definition}

\begin{proposition}\label{oplax-sen-prop}
Dfn.~\ref{sen-dfn} gives a oplax $\3/2$-functor 
$p\Sen \co p\Sign \ra \Setp$.  
\end{proposition}

\begin{proof}
Note that $\Sen$ and $p\Sen$ are the same on the signatures, they
differ only on the signature morphisms. 
The oplax property of $p\Sen$ on the identities is rather immediate;
in fact it holds in the strict form $p\Sign (1_\Sigma) =
1_{\Sen(\Sigma)}$. 

Let us now focus on proving that 
\begin{equation}\label{funct-equation}
p\Sen(\varphi;\varphi') \subseteq p\Sen(\varphi);p\Sen(\varphi').
\end{equation}

We consider $\rho\in \DOM \ p\Sen(\varphi;\varphi')$ which by
Dfn.~\ref{sen-dfn} means $\rho\in \Sen(\DOM \ \varphi;\varphi')$. 
\begin{itemize}

\item Since by Dfn.~\ref{partial-def} we have that $\DOM \
  \varphi;\varphi' \subseteq \DOM \varphi$ and because $\I$ is
  inclusive, we get that $\rho\in \Sen(\DOM \varphi) =
  \DOM \ p\Sen(\varphi)$.  

\item By the commutativity of the square ($\Diamond$) of
  Dfn.~\ref{partial-def} we have that $\Sen(\varphi^0)\rho =
  \Sen({\varphi^0}')\rho \in \Sen(\DOM \varphi') = \DOM \
  p\Sen(\varphi')$.   

\end{itemize}
This means $\rho \in \DOM \ p\Sen(\varphi);p\Sen(\varphi')$. 
Hence $\DOM \ p\Sen(\varphi;\varphi') \subseteq \DOM \
p\Sen(\varphi);p\Sen(\varphi')$. 

For any $\rho\in \Sen(\DOM \varphi;\varphi') = \DOM \
p\Sen(\varphi;\varphi')$  we have the following:
\[
\begin{array}{rll}
p\Sen(\varphi;\varphi')\rho = &
  \Sen((\varphi;\varphi')^0)\rho 
  & \quad \text{by the definition of }p\Sen(\varphi;\varphi')\\[.2em]
= & \Sen({\varphi^0}';\varphi'^0)\rho 
  & \quad \text{by the definition of }(\varphi;\varphi')^0 
     \text{ cf. diagram \eqref{diag-comp}}\\[.2em]
= & \Sen(\varphi'^0)(\Sen({\varphi^0}')\rho) 
  & \quad \text{by the functoriality of }\Sen \\[.2em]
= & \Sen(\varphi'^0)(\Sen(\varphi^0)\rho) 
  & \quad \text{by applying }\Sen \text{ to the square } (\Diamond)  
     \text{ of Dfn.~\ref{partial-def}} \\[.2em]
= & p\Sen(\varphi')(p\Sen(\varphi)\rho) 
  & \quad \text{by the definition of }p\Sen(\varphi), p\Sen(\varphi')\\[.2em]
= & (p\Sen(\varphi);p\Sen(\varphi'))\rho. 
  &
\end{array}
\]
This concludes the proof of \eqref{funct-equation} and of the
proposition. 
\end{proof}

In many concrete situations of interest in fact the
sentence$\3/2$-functor $p\Sen$ is strict.
The following result gives a widely applicable general condition for
that.

\begin{corollary}\label{lax-sen-cor}
If $\Sen$ maps each pullback square of semi-inclusive cospans 
to a weak pullback square, then $p\Sign$ is a strict $\3/2$-functor. 
\end{corollary}

\begin{proof}
By Prop.~\ref{oplax-sen-prop} it is enough to prove that $p\Sen$ is
lax. 
Since $p\Sign$ is strict on the identities anyway we prove only that  
\begin{equation}\label{lax-funct-equation}
p\Sen(\varphi);p\Sen(\varphi') \subseteq p\Sen(\varphi;\varphi').
\end{equation}
For this we consider $\rho \in \DOM \ p\Sen(\varphi);p\Sen(\varphi')$. 
This means $\rho \in \DOM \ p\Sen(\varphi) = \Sen(\DOM\varphi)$ and
furthermore that $\Sen(\varphi^0)\rho \in \DOM \ p\Sen(\varphi') =
\Sen(\DOM \varphi')$. 
By the hypothesis when we apply $\Sen$ to the square ($\Diamond$) of
Dfn.~\ref{partial-def} we still get a pullback square, which means
that there exists a sentence in $\Sen(\DOM \varphi;\varphi')$
that gets mapped to $\rho$ by the inclusion $\Sen(\DOM
\varphi;\varphi')\subseteq \Sen(\DOM \varphi)$ and to 
$\Sen(\varphi^0)\rho$ by $\Sen({\varphi^0}')$; of course because of the
inclusion $\Sen(\DOM \varphi;\varphi')\subseteq \Sen(\DOM \varphi)$
this sentence must be $\rho$.
Hence we have just proved that 
$\DOM \ p\Sen(\varphi);p\Sen(\varphi')\subseteq 
\DOM \ p\Sen(\varphi;\varphi')$.
The rest of the argument is similar to the corresponding part from the
proof of Prop.~\ref{oplax-sen-prop}. 
\end{proof}

\begin{example}
\begin{rm}
The sentence functors of both $\3/2 \PL$ and $\3/2 \MSA$ arise as
instances of Dfn.~\ref{sen-dfn}. 
While in both cases the existence of pullbacks in $\Sign$ is easy or
well known (see \cite{iimt,fun3} for the $\MSA$), the assumption on
$\Sen$ being inclusive deserves here a bit of attention.  

While the sentence functor $\Sen^{\PL}$ of $\PL$ (propositional
logic) is obviously inclusive and while for the  definitions of $\MSA$
that take a global approach to quantification variables this is true
as well (such as in
\cite{Goguen-Burstall:Institutions-1992,iimt,sannella-tarlecki-book},
etc.), when considering a local approach to quantifiers (like in 
some more recent publications; see \cite{UniLogCompSci} for an ample
discussion on the issue) the sentence functor is not inclusive
anymore, signature inclusions being mapped to designated injections
that are subject to some coherence properties. 
This is basically due to the fact that in the local approach to
quantification variables these are rather heavily qualified, very much
like in the implementations of specification languages (e.g. CafeOBJ
\cite{caferep}), and for example the qualifications by the signatures
are not preserved by sentence translations along inclusions. 
In such cases the solution has been already formulated above, namely
to weaken the concept of inclusion system to a system of designated
``abstract injections''. 
However for the sake of simplicity of presentation we stick here
with the concept of inclusive functor, but keeping in the mind that
for the situations when this does not really work there is exists
technical remedy.

In both the $\3/2 \PL$ and the $\3/2 \MSA$ cases the sentence functors
are strict, which means they are are also lax $\3/2$-functors.  
This is due to the fact the condition of preservation of pullbacks of
Cor.~\ref{lax-sen-cor} holds both in $\PL$ and $\MSA$. 
In fact it holds often in concrete institutions even in a stronger
form (for all pullbacks \cite{iimt}, exercise 4.19, called there
\emph{weak coamalgamation}).  
We will show how this works in the case of $\PL$, the following
argument being easily replicated to other situations. 

Let us consider a pullback square of a semi-inclusive cospan in
$\Sign^\PL$ (depicted as the square $(*)$ in the diagram below).
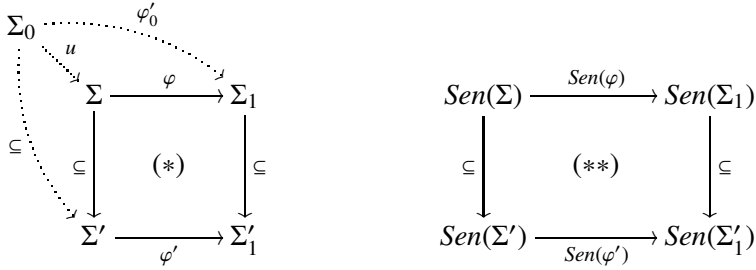
\[
\xymatrix @-1.2em{
\Sigma_0 \ar@{.>}[dr]^u \ar@{.>}@/^1em/[drrr]^{\varphi'_0} 
\ar@{.>}@/_1em/[dddr]_\subseteq & & & &
& & & \\
  & \Sigma \ar[dd]_\subseteq \ar[rr]^\varphi &  & \Sigma_1 \ar[dd]^\subseteq
  & & & & \Sen(\Sigma) \ar[rr]^{\Sen(\varphi)} \ar[dd]_{\subseteq} & &
  \Sen(\Sigma_1) \ar[dd]^{\subseteq} \\
  & & (*) & & & & & & (**) & \\
  & \Sigma' \ar[rr]_{\varphi'} & & \Sigma'_1 & & & & \Sen(\Sigma')
  \ar[rr]_{\Sen(\varphi')} & & \Sen(\Sigma'_1)
}
\]
That the square $(*)$ is a pullback means that 
$\Sigma = \{ p \in \Sigma' \mid \varphi'(p) \in \Sigma_1 \}$.
That the square $(**)$ is a weak pullback means that for any $\rho_1
\in \Sen(\Sigma_1)$ and $\rho' \in \Sen(\Sigma')$ such that
$\Sen(\varphi')\rho' = \rho_1$ we have that $\rho' \in \Sen(\Sigma)$
and $\Sen(\varphi)\rho' = \rho_1$. 
Note that in this case, because of the inclusions, weak pullback means
just pullback.  

Consider $\sigma_0 \subseteq \Sigma'$ to be the set of symbols
occurring in $\rho'$. 
Because $\Sen(\varphi')\rho' \in \Sen(\Sigma_1)$ it follows that the
restriction of $\varphi'$ to $\Sigma_0$ maps everything to
$\Sigma_1$; we denote it by $\varphi'_0$.
By the pullback property of $(*)$ there exists an unique $u$ such that
$u ; (\Sigma \subseteq \Sigma') = (\Sigma_0 \subseteq \Sigma')$. 
It follows that $u$ is inclusion too. 
Hence $\rho'\in \Sen(\Sigma_0) \subseteq \Sen(\Sigma)$. 
Also $\Sen(\varphi)\rho' = \Sen(\varphi'_0)\rho' = \Sen(\varphi')\rho'
= \rho_1$. 
\end{rm}
\end{example}

\subsection{The model functor $p\Mod$ and the Satisfaction Condition}

\begin{definition}\label{mod-dfn}
Given any functor $\Mod \co \Sign^\varominus \to \CAT$ we define  
\begin{itemize}

\item for each $\Sigma\in |\Sign| \ (= |p\Sign|)$, 
\ \ $p\Mod(\Sigma) = \Mod(\Sigma)$, 

\item for each partial $\Sign$-morphism $\varphi \co \Sigma \pto
  \Sigma'$, 
\[
p\Mod(\varphi)M' = \{ M \in |\Mod(\Sigma)| \ \mid \ \Mod(\DOM \varphi
\subseteq \Sigma)M = \Mod(\varphi^0)M' \}.
\] 

\item $p\Mod(\varphi)$ is defined on the arrows  like on the
  objects.  

\end{itemize}
\end{definition}

\begin{proposition}\label{mod-prop}
$p\Mod$ is a lax functor 
$(p\Sign)^{\varoplus}\rightarrow \3/2 (\CAT_{\!\!\P})$.
\end{proposition}

\begin{proof}
First we show that for each partial $\Sign$-morphism $\varphi\co
\Sigma \to \Sigma'$, the model reduct $\3/2$-functor
$p\Mod(\varphi) \co \Mod(\Sigma') \to \P \Mod(\Sigma)$ is lax. 
We denote the inclusion $\DOM\varphi \subseteq \Sigma$ by $d_\varphi$.

For any model homomorphisms $f',g'\in p\Mod(\Sigma')$ with $\cod{f'}=
\dom{g'}$ the homomorphisms in $p\Mod(\varphi)f' ; \Mod(\varphi)g'$
are $f;g$ with $f,g\in \Mod(\Sigma)$ such that $\cod{f} = \dom{g}$, 
$\Mod(d_\varphi)f = \Mod(\varphi^0)f'$, and 
$\Mod(d_\varphi)g = \Mod(\varphi^0)g'$.  
The calculation 
\[
\begin{array}{rll}
\Mod(d_\varphi)(f;g) = 
  & \Mod(d_\varphi)(f);
    \Mod(d_\varphi)(g)
  & \quad \text{by the functoriality of } \Mod(d_\varphi) \\[.2em] 
= & \Mod(\varphi^0)f' ; \Mod(\varphi^0)g'
  & \quad \text{since } f\in p\Mod(\varphi)f' \text{ and } g\in
  p\Mod(\varphi)g' \\[.2em]
= & \Mod(\varphi^0)(f';g')
  & \quad \text{by the functoriality of }\Mod(\varphi^0)
\end{array}
\]
shows that $f;g\in p\Mod(\varphi)(f';g')$.
Hence  $p\Mod(\varphi)f' ; \Mod(\varphi)g' \subseteq
p\Mod(\varphi)(f';g')$. 

Now we show that $p\Mod$ is a lax $\3/2$-functor. 
\begin{itemize}

\item Let $\varphi \leq \theta$ be partial $\Sign$-morphisms, which
  means the diagram below commutes:
\[
\xymatrix{
\dom{\varphi} \ar@{=}[d]  & \DOM \varphi \ar[l]_\subseteq
\ar[r]^{\varphi^0} \ar[d]_\subseteq & \cod{\varphi} \ar@{=}[d] \\
\dom{\theta}  & \DOM\theta\ar[r]_{\theta^0} \ar[l]^\subseteq &
\cod{\theta} 
}
\]
For any $M \in p\Mod(\theta)M'$ we have that 
\[
\begin{array}{rll}
\Mod(\varphi^0)M' = & \Mod(\DOM \varphi \subseteq \DOM \theta) (\Mod(\theta^0)M')
  & \\[.2em]
= & \Mod(\DOM \varphi \subseteq \DOM \theta)(\Mod(\DOM \theta
    \subseteq \dom{\theta})M) & \\[.2em]
= & \Mod(\DOM \varphi \subseteq \dom{\varphi})M. & 
\end{array}
\]
Hence $M \in p\Mod(\varphi)M'$ which shows that $p\Mod(\theta)M'
\subseteq p\Mod(\varphi)M'$. 

\item
We consider partial $\Sign$-morphisms $\varphi$, $\varphi'$ such that
$\cod{\varphi} = \dom{\varphi'}$ and consider 
$M \in p\Mod(\varphi)(p\Mod(\varphi')M'')$. 
This means that there exists $M'$ such that $M \in p\Mod(\varphi)M'$
and $M' \in p\Mod(\varphi')M''$. 
We have that 
\[
\hspace{-2em}
\begin{array}{rll}
\Mod((\varphi;\varphi')^0)M'' = & \Mod((\varphi^0)' ; \varphi'^0)M'' &
  \quad \text{from diagram \eqref{diag-comp}}\\[.2em]
= & \Mod({\varphi^0}')(\Mod(\varphi'^0)M'') & \quad \text{by the
                                            functoriality of } \Mod \\[.2em]
= & \Mod({\varphi^0}')(\Mod(\DOM \varphi \subseteq \dom{\varphi'})M'') &
  \quad \text{since }M' \in p\Mod(\varphi')M' \\[.2em]
= & \Mod({\varphi^0}';(\DOM \varphi \subseteq \dom{\varphi'}))M' & 
\quad \text{by the functoriality of }\Mod \\[.2em]
= & \Mod((\DOM\varphi;\varphi' \subseteq
    \DOM\varphi);\varphi^0)M') & \quad \text{from diagram
                                 \eqref{diag-comp}} \\[.2em] 
= & \Mod(\DOM\varphi;\varphi' \subseteq
    \DOM\varphi)(\Mod(\varphi^0)M') & \quad \text{by the
                                            functoriality of } \Mod \\[.2em]
= & \Mod(\DOM\varphi;\varphi' \subseteq
    \DOM\varphi)(\Mod(\DOM \varphi \subseteq \dom{\varphi})M) & 
  \quad \text{since } M \in \Mod(\varphi)M' \\[.2em]
= & \Mod((\DOM\varphi;\varphi' \subseteq
    \DOM\varphi);(\DOM \varphi \subseteq \dom{\varphi}))M & 
  \quad \text{by the functoriality of } \Mod \\[.2em]
= & \Mod(\DOM\varphi;\varphi' \subseteq \dom{\varphi;\varphi'})M. & 
\end{array}
\]
This shows that $M\in \Mod(\varphi;\varphi')M''$, hence
$p\Mod(\varphi');p\Mod(\varphi) \leq p\Mod(\varphi;\varphi')$. 

\item We know that the identities in $p\Sign$ are $[1_\Sigma]$ where
  $1_\Sigma$ is an identity in $\Sign$. 
  From the definition, it follows that 
  $p\Mod([1_\Sigma])M = \{ M \}$, hence $M \in p\Mod([1_\Sigma])M$.  

\end{itemize}
\end{proof}

\begin{example}
\begin{rm}
Both the model functors of $\3/2 \PL$ and of $\3/2 \MSA$ arise as
instances of Dfn.~\ref{mod-dfn} (as $p\Mod$, where $\Mod$ is the model
functor of $\PL$ and $\MSA$, respectively). 
The same holds for the model functors of the $\3/2$-subinstitutions of
$\MSA$ discussed in Ex.~\ref{3/2-submsa-ex}.
\end{rm}
\end{example}

Now we are able to define a $\3/2$-institution on the basis of any
institution (with its category of signatures satisfying the required
technical assumptions) by putting together the results of the
Propositions \ref{oplax-sen-prop} and \ref{mod-prop}. 
Since both the sentences and the models of the $\3/2$-institution are
exactly those of the base institution, the satisfaction relation of
the $\3/2$-institution is inherited from the base institution.  

\begin{corollary}
For any institution $\I = (\Sign,\Sen,\Mod,\models)$ such that 
\begin{itemize}

\item $\Sign$ is endowed with an inclusion system, 

\item $\Sign$ has pullbacks of semi-inclusive cospans, and  

\item $\Sen$ is an inclusive functor, 

\end{itemize}
$\3/2 \I = (p\Sign,p\Sen,p\Mod,\models)$ is an oplax
$\3/2$-institution. 
\end{corollary}

\begin{proof}
By the virtue of the Propositions \ref{oplax-sen-prop} and
\ref{mod-prop} it only remains to show the Satisfaction Condition for  
$\3/2 \I$. 
We consider $\varphi \in p\Sign$, $\rho\in \DOM \ p\Sen(\varphi)$, and
models $M,M'$ such that $M \in p\Mod(\varphi)M'$. 
Then
\[
\begin{array}{rll}
M' \models_{\cod{\varphi}} p\Sen(\varphi)\rho \quad \text{if and only if}
   &
   M'\models_{\cod{\varphi}} \Sen(\varphi^0)\rho & \quad \text{by definition
                                                   of } p\Sen \\[.2em]
\text{if and only if} & \Mod(\varphi^0)M' \models_{\DOM \varphi} \rho & 
  \quad \text{by the Satisfaction Condition in } \I \\[.2em]
\text{if and only if} & \Mod(\DOM \varphi \subseteq \dom{\varphi})M
                        \models_{\DOM \varphi} 
    \rho & \quad \text{by definition of  } p\Mod(\varphi)M' \\[.2em]
\text{if and only if} & M \models_{\dom{\varphi}} \Sen(\DOM \varphi
                        \subseteq 
    \dom{\varphi})\rho & \quad \text{by the Satisfaction Condition in
                         } \I \\[.2em]
\text{if and only if} & M \models_{\dom{\varphi}} \rho & 
      \quad \text{since }\Sen \text{ is inclusive.} 
\end{array}
\]
\end{proof}

\section{Properties of $\3/2 \I$}

In this section we determine some properties of $\3/2 \I$ that are
significant within the context of the general theory of
$\3/2$-institutions and its envisaged applications as developed in
\cite{3/2Inst}.  

\subsection{Total signature morphisms} 

In \cite{3/2Inst} a couple of complementary concepts, one of syntactic
and the other of semantic nature, have been introduced in order to
reflect abstractly the situation when a signature morphism is total.  
These concepts have been applied in \cite{3/2Inst} for deriving
crucial properties on colimits and model amalgamation. 
 
A signature morphism $\varphi$ in a $\3/2$-institution
$(\Sign,\Sen,\Mod,\models)$ is 
\begin{itemize}

\item \emph{$\Sen$-maximal} when $\Sen(\varphi)$ is total;

\item \emph{$\Mod$-maximal} when for each $\cod{\varphi}$-model $M'$,
  $\Mod(\varphi)M'$ is a singleton; and 

\item \emph{total} when it is both $\Sen$-maximal and $\Mod$-maximal. 

\end{itemize}
The following is an expected straightforward property. 

\begin{fact}
For any signature morphism $\chi$ in $\I$, $[\chi]$ is total in 
$\3/2 \I$.  
\end{fact}

The following is yet another semantic technical expression of the
totalness of the signature morphisms that has been used in
\cite{3/2Inst} in connection to model amalgamation properties. 

In a $\3/2$-institution a signature morphism $\varphi$ is
\emph{$\Mod$-strict} when for each signature morphism $\theta$ such
that $\cod{\theta} = \dom{\varphi}$ we have that 
\[
\Mod(\varphi);\Mod(\theta)=\Mod(\theta;\varphi).
\]

In general, in many concrete situations of interest -- $\3/2 \PL$ and
$\3/2 \MSA$ included -- a signature morphism is $\Mod$-strict whenever 
it is total. 

\begin{proposition}
For any signature morphism $\chi$ in $\I$, $[\chi]$ is $p\Mod$-strict
(in  $\3/2 \I$).   
\end{proposition}

\begin{proof}
Let $\chi \in \Sign$ (the category of the signatures of $\I$).
Let $\varphi\in p\Sign$ such that $\cod{\varphi} = \dom{\chi}$. 
We have to show that 
\[
p\Mod([\chi]);p\Mod(\varphi) = 
p\Mod(\varphi;[\chi]).
\]
Since $p\Mod$ is lax we have to show only that the right hand side is
less than the left hand side. 
Therefore let $M'' \in \Mod(\cod{\chi})$ and let 
$M \in p\Mod(\varphi;[\chi])M''$. 
We have to show that $M \in
p\Mod(\varphi)(p\Mod([\chi])M'')$.  

Note that 
\begin{equation}\label{strict-eq1}
p\Mod([\chi])M'' = \{ \Mod(\chi)M'' \}
\end{equation}
and that the collapse of the square $(\Diamond)$ in diagram
\eqref{diag-comp} to $\varphi^0$ means also that   
\begin{equation}\label{strict-eq2} 
\DOM (\varphi;[\chi]) = \DOM \varphi
\mbox{ \ and \ } 
(\varphi;[\chi])^0 = \varphi^0;\chi. 
\end{equation}
By \eqref{strict-eq2} and by the definition of $p\Mod$, 
$M \in p\Mod(\varphi;[\chi])M''$ translates to 
\[
\Mod(\varphi^0)(\Mod(\chi)M'') = 
\Mod(\DOM \varphi \subseteq \dom{\varphi})M.
\]
which means that $M \in p\Mod(\varphi)(\Mod(\chi)M'')$.
By \eqref{strict-eq1} it follows that $M \in
p\Mod(\varphi)(p\Mod([\chi])M'')$. 
\end{proof}

\subsection{Lax cocones and model amalgamation}

The main proposal of \cite{3/2Inst} regarding the $\3/2$-institution
theoretic foundations of conceptual blending is based upon two
concepts: lax cocones and model amalgamation. 
Both of them constitute $\3/2$-institution theoretic extension of
corresponding ordinary institution theoretic concepts. 
In this section we develop a result on the existence of lax cocones
and model amalgamation in $\3/2 \I$ based upon existence of cocones
and model amalgamation in the base institution $\I$. 
By considering the mere fact that these properties are very common in
concrete institutions, this result is applicable to a wide range of
concrete situations. 

We start by recalling the well established notion of model
amalgamation in ordinary institution theory, then we move to recalling
its $\3/2$-institution theoretic extension from \cite{3/2Inst} and
finally we develop the above mentioned result. 

Model amalgamation properties for institutions formalize the possibility of
amalgamating models of different signatures when they are
consistent on some kind of generalized `intersection' of signatures.
It is one of the most pervasive properties of concrete institutions
and it is used in  a crucial way in many institution theoretic
studies.
A few early examples are 
\cite{sannella-tarlecki88,tarlecki86,jm-granada89,modalg}. 
For the role played by this property in specification theory and in
institutional model theory see \cite{sannella-tarlecki-book} and
\cite{iimt}, respectively.

A \emph{model of a diagram of signature morphisms in an institution}
consists of a model $M_k$ for each signature $\Sigma_k$ in the
diagram such that for each signature morphism $\varphi \co \Sigma_i
\to \Sigma_j$ in the diagram we have that $M_i = \Mod(\varphi)M_j$.    

A commutative square of signature morphisms
$$\xy
\xymatrix { 
\Sigma \ar[r]^{\varphi_1}   \ar[d]_{\varphi_2}
    & \Sigma_1       \ar[d]^{\theta_1} \\
\Sigma_2  \ar[r]_{\theta_2}
    & \Sigma'
}
\endxy$$
is an \emph{amalgamation square} if and only if each model of the span
$(\varphi_1,\varphi_2)$ admits an unique completion to a model of the
square. 
When we drop off the uniqueness requirement we call this a 
\emph{weak model amalgamation square}.

In most of the institutions formalizing conventional or non-conventional 
logics, pushout squares of signature morphisms are model amalgamation
squares  \cite{iimt}.  

In the literature there are several more general concepts of model
amalgamation. 
One of them is model amalgamation for cocones of arbitrary diagrams
(rather than just for spans), another one is model amalgamation for
model homomorphisms. 
Both are very easy to define by mimicking the definitions presented
above. 
While the former generalisation is quite relevant for the intended
applications of our work, the latter is less so since at this moment
model homomorphisms do not seem to play any role in conceptual
blending or in merging of software changes. 
Moreover amalgamation of model homomorphisms is known to play a role
only in some developments in institution-independent model theory 
\cite{iimt}, but even there most involvements of model amalgamation
refers only to amalgamation of models.  

In \cite{3/2Inst} this notion is extended to $\3/2$-institutions.
For the sake of simplicity of presentation, this is presented for lax
cocones of spans, the general concept for lax cocones over arbitrary
diagrams of signature morphisms being an obvious generalisation.  

A \emph{model for a diagram of signature morphisms} in a
$\3/2$-institution consists of a model $M_k$ for each signature
$\Sigma_k$ in the diagram such that for each signature morphism
$\varphi \co \Sigma_i \to \Sigma_j$ in the diagram we have that 
$M_i \in \Mod(\varphi)M_j$. 

The diagram is \emph{consistent} when it has at least one model. 

In any $\3/2$-institution, a lax cocone for a span in the
$\3/2$-category of the signature morphisms  
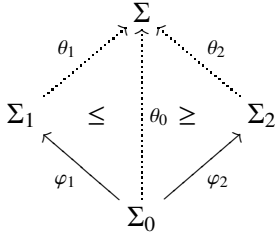
\[
\xymatrix @C-2em{
 & & \Sigma & & \\
\Sigma_1 \ar@{.>}[urr]^{\theta_1} & { \ \ \ \ \leq} & &
{ \geq \ \ \ \ } &
\Sigma_2 \ar@{.>}[ull]_{\theta_2}\\
 &  & \Sigma_0 \ar[ull]^{\varphi_1} \ar[urr]_{\varphi_2} \ar@{.>}[uu]_{\theta_0}& &
}
\]
has \emph{model amalgamation} when each model of the span admits an
unique completion to a model (called the \emph{amalgamation}) of the
lax cocone.  

When dropping the uniqueness condition, the property is called
\emph{weak model amalgamation}.  

\begin{proposition}\label{lax-cocone-prop}
Let $\I$ be an inclusive institution with pullbacks of semi-inclusive
cospans. 
\begin{enumerate}

\item If each span of signature morphisms in $\I$ admits a cocone 




then each span of signature morphisms $(\varphi_1,\varphi_2)$ in 
$\3/2 \I$ admits a lax cocone. 
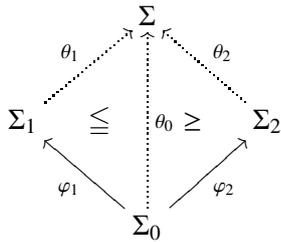
\[
\xymatrix @C-2em{
 & & \Sigma & & \\
\Sigma_1 \ar@{.>}[urr]^{\theta_1} & { \ \ \ \ \leqq} & &
{ \geq \ \ \ \ } &
\Sigma_2 \ar@{.>}[ull]_{\theta_2}\\
 &  & \Sigma_0 \ar[ull]^{\varphi_1} \ar[urr]_{\varphi_2}
 \ar@{.>}[uu]_{\theta_0}& & 
}
\]

\item If each span of signature morphisms in $\I$ admits a cocone
that has (weak) model amalgamation, then each span of signature
morphisms in $\3/2 \I$ admits a lax cocone that has (weak) model
amalgamation. 

\end{enumerate}
\end{proposition}

\begin{proof}
1.
We define the a lax cocone for $(\theta_0,\theta_1,\theta_2)$ for the
span $(\varphi_1,\varphi_2)$ as follows.
We consider any $\DOM \theta_0 \subseteq \Sigma_0$ such that 
$\DOM \theta_0 \subseteq \DOM \varphi_k$, $k=1,2$. 
Note that in general there may be several choices for $\DOM \theta_0$,
and definitely at least one exists by letting $\DOM \theta_0 = \Sigma_0$. 
In $\I$ we consider successively:
\begin{itemize}

\item for $k = 1,2$, $(u_k,v_k)$ a cocone for the span $(\varphi_k^0,
  \DOM \varphi_k \subseteq \DOM \theta_0)$, and  

\item $(w_1,w_2)$ a cocone for the span $(v_1,v_2)$.

\end{itemize}
For $k=1,2$ let $\theta_k^0 = u_k ; v_k$. 
\begin{equation}\label{lax-cocone-diag}
\xymatrix @C-1em {
 & & \Sigma & & \\
 & \Sigma'_1 \ar@{.>}[ur]_{w_1} & & \Sigma'_2 \ar@{.>}[ul]^{w_2} & \\
\Sigma_1 \ar@{.>}[ur]_{u_1} \ar@/^2em/@{.>}[uurr]^{\theta^0_1} & &
\DOM \theta_0 \ar@{.>}[ul]^{v_1} \ar@{.>}[uu]_{\theta^0_0} \ar@{.>}[ur]_{v_2}
\ar[dd]_{\subseteq} & & \Sigma_2 \ar@{.>}[ul]^{u_2} \ar@/_2em/@{.>}[uull]_{\theta^0_2}\\
 & \DOM \varphi_1 \ar[ul]_{\varphi^0_1} \ar[ur]^{\subseteq}
 \ar[dr]_{\subseteq}
 \ar `l/40pt[u] `[uuu]^{(\varphi_1;\theta_1)^0} [uuur]& & 
 \DOM \varphi_2 \ar[ur]^{\varphi^0_2}
 \ar[ul]_{\subseteq}\ar[dl]^{\subseteq}
 \ar `r/40pt[u] `[uuu]_{(\varphi_2;\theta_2)^0} [uuul]& \\
 & & \Sigma_0 & & 
}
\end{equation}
Then $\theta^0_k$, $k = 0,1,2$ define the $\3/2 \I$ signature morphisms 
$\theta_0 \co \Sigma_0 \pto \Sigma$, 
$\theta_k = [\theta^0_k] \co \Sigma_k \pto \Sigma$, $k=1,2$. 
So $\theta_1$ and $\theta_2$ are total, which implies
that for $k=1,2$,  
\[
\DOM(\varphi_k ; \theta_k) = \DOM\varphi_k
\mbox{ \ and \ }
(\varphi_k ; \theta_k)^0 = \varphi_k^0;\theta_k^0.
\]
By chasing diagram \eqref{lax-cocone-diag} we establish that for
$k=1,2$, $\varphi^0_k ; \theta^0_k = 
(\DOM \varphi_k \subseteq \DOM \theta_0) ; \theta^0_0$.
It thus follows that $\varphi_k ; \theta_k \leq \theta_0$.  

2. 
For the second part of the proposition we consider
models $M_0, M_1, M_2$ such that for $k=1,2$, 
$M_0 \in p\Mod(\varphi_k)M_k$. 
When considering the three cocones for the spans of signature morphisms
in $\I$ as above, we may also consider them to admit weak model
amalgamation.
  
For $k=1,2$ let $N_k = \Mod(\DOM \varphi_k \subseteq \Sigma_0)M_0$ and
let $N_0 = \Mod(\DOM \theta_0 \subseteq \Sigma_0)M_0$.
For $k=1,2$ we have that $\Mod(\varphi^0_k)M_k = N_k = \Mod(\DOM
\varphi_k \subseteq \DOM \theta_0)N_0$.
\begin{itemize}

\item for $k=1,2$ let  $M'_k$ be the amalgamation of $N_0$ and $M_k$
  (in $\I$, for the cocone $(u_k,v_k)$ of the span $(\varphi_k^0,
  \DOM \varphi_k \subseteq \DOM \theta_0)$), 
  and 

\item let $M$ be the amalgamation of $M'_1$ and $M'_2$ (in $\I$, for
  the cocone $(w_1,w_2)$ of the span $(v_1,v_2)$). 

\end{itemize}
Then $M$ is the amalgamation of $M_0, M_1, M_2$ for the lax cocone
$(\theta_0,\theta_1,\theta_2)$ of the span $(\varphi_1,\varphi_2)$. 

For the strict (non-weak) case, it is enough to note that 
\begin{itemize}

\item $N_0$ is uniquely determined by the equation $\Mod(\DOM \theta_0
  \subseteq \Sigma_0)M_0 = N_0$ (which is a consequence of 
$M_0 \in p\Mod(\theta_0)M$),

\item for $k=1,2$, $M'_k$ are uniquely determined by $M_k$ (that are
  given) and $N_0$, and 

\item $M$ is uniquely determined by $M'_1$ and $M'_2$. 

\end{itemize}
\end{proof}
The condition that each span of signature morphisms admits a cocone
is usually easy in concrete situations, for example the stronger
version of existence of pushout cocones holds in most situations of
interest.   
Very often pushout cocones have model amalgamation property at least
in its weak form (see for example
\cite{iimt,sannella-tarlecki-book}). 

Note that in principle there are several choices for $\DOM \theta_0$. 
Each of them determines a different lax cocone of signature morphisms
in $\3/2 \I$.  
Another parameter that may lead to different constructed variants of
the lax cocone is the choice of the cocones in the category of the
signatures of $\I$; for example in the weak model amalgamation case, in
general there can be multiple choices.


\subsection{A taxonomy of ``partial'' theory morphisms}

In \cite{3/2Inst} there is a study of how the concept of theory can be
extended to the $\3/2$-institutional framework, which is motivated by
the applications to conceptual blending.
In Goguen's approach to conceptual blending
\cite{Goguen:Algebraic-Semiotics-1999,Goguen:Mathematical-Models-of-Cognitive-Space-and-Time-2006}
concept translation is modelled as translation of logical theories
(but this has not been achieved there in a proper way due to the lack
of an institution theoretic framework). 
While theories in $\3/2$-institutions are the same as theories in
1-institutions, the $\3/2$-institution theoretic concept of theory
morphism is much more subtle because of the partiality of the sentence
translations.    
In \cite{3/2Inst} it is shown how mathematically there are four
possible extensions of the institution theoretic concept of theory
morphism to $\3/2$-institutions but also that only the following two
of these make real sense technically and in the applications.  

In a $\3/2$-institution given two theories $(\Sigma,E)$ and
$(\Sigma',E')$, a signature morphism 
$\varphi \co \Sigma \to \Sigma'$ is  
\begin{itemize}

\item a \emph{weak $\3/2$-theory morphism} $(\Sigma,E) \to
  (\Sigma',E')$ when $\Sen(\varphi)E^\bullet
  \subseteq E'^\bullet$,  

\item a \emph{strong $\3/2$-theory morphism} $(\Sigma,E) \to
  (\Sigma',E')$ when for each $\Sigma'$-model $M'$ such that $M'
  \models E'$ there exists $M \in \Mod(\varphi)M'$ such that $M
  \models E$.

\end{itemize}

The relationship between these two concepts is an inclusive one:

\begin{fact}\cite{3/2Inst}\label{weak-strong-fact}
Each strong $\3/2$-theory morphism is weak. 
\end{fact}

In the light of the above considerations, given an institution $\I$
endowed with an inclusion system and with pullbacks of semi-inclusive
cospans, there are four ways to think about  partiality for theory
morphisms.  
These four can be grouped as follows: 
\begin{enumerate}

\item Consider the closed and the strong inclusion systems for
  theories in $\I$; each of them determines a concept of partial
  theory morphisms by considering the category of closed theories
  in the role of $\Sign$ in Dfn.~\ref{partial-def}.
  For this it is necessary to recall (from \cite{iimt}) that in any
  institution the pullbacks of (closed) theories are inherited from the
  underlying category of the signatures.
  Let us call the former concept \emph{closed-partial theory morphism}
  and the latter one \emph{strong-partial theory morphism}.  

\item Consider $\3/2 \I$, the $\3/2$-institution built on top of
  $\I$. 
  Then we may consider the two significant concepts of $\3/2$-theory
  morphisms discussed above, the weak and the strong one. 

\end{enumerate}

In the following we analyse and establish the relationships between
these four concepts of partiality for theory morphisms, a first step
being given by Fact \ref{weak-strong-fact}. 
The following is a crucial result in this direction:

\begin{proposition}
Given an inclusive institution $\I = (\Sign,\Sen,\Mod,\models)$ with 
pullbacks of semi-inclusive cospans, then the category of weak
$\3/2$-theory morphisms in 
$\3/2 \I = (p\Sign, p\Sen, p\Mod,\models)$ is equivalent to the
category of closed-partial theory morphisms.  
\end{proposition}

\begin{proof}
Any weak $\3/2$-theory morphisms in $\3/2 \I$, 
$\varphi \co (\Sigma,E) \to (\Sigma',E')$, gets mapped to the
closed-partial theory morphism 
$\overline{\varphi} \co (\Sigma,E^\bullet) \pto (\Sigma',E'^\bullet)$
\[
\xymatrix @C-1em{
(\Sigma,E^\bullet) \ar@{.>}[rr]^{\overline{\varphi}} & & (\Sigma',E'^\bullet) \\
 & \DOM \overline{\varphi}
   \ar[ul]^\subseteq \ar[ur]_{(\overline{\varphi})^0} & 
}
\]
where  
\begin{itemize}

\item $\DOM \overline{\varphi} = (\DOM \varphi, E^\bullet \cap
  \Sen(\DOM \varphi))$, 

\item $(\overline{\varphi})^0 = \varphi^0$. 

\end{itemize}
That $(\overline{\varphi})^0$ is a theory morphism in $\I$ follows
from 
\[
\Sen(\varphi^0)(E^\bullet \cap \Sen(\DOM \varphi)) \subseteq E'^\bullet
\]
which can be expressed as 
\[
p\Sen(\varphi)E^\bullet \subseteq E'^\bullet
\]
which is exactly the condition that $\varphi$ is a weak $\3/2$-theory
morphism. 

It is rather straightforward to check that this mapping $\varphi
\mapsto \overline{\varphi}$ is functorial and that it is an
equivalence of categories. 
\end{proof}

\begin{corollary}
The following figure shows the relationships between the four concepts
of ``partial'' theory morphisms:

\

\begin{rm}
\nested{%
  strong 3/2-theory morphism%
}{%
  weak 3/2-theory morphism = closed-partial theory morphism%
}{%
  strong-partial theory morphism%
}
\end{rm}
\end{corollary}




\section*{References}

\bibliographystyle{plain}
\bibliography{/Users/diacon/TEX/tex,ERC2016references}

\end{document}